\RequirePackage{pgfplots}
\RequirePackage{ifthen} 
\PassOptionsToPackage{usenames,dvipsnames,svgnames}{xcolor}

\def\mode{PreprintMode}

\def\JournalMode{
	\documentclass[sn-mathphys]{sn-jnl}
}
\def\PreprintMode{
	\documentclass{article}
	\pdfoutput=1
}

\csname\mode\endcsname

\usepackage[T1]{fontenc}

\usepackage{pdfsync}
\usepackage{float}

\usepackage{latexsym}

\usepackage[mathscr]{euscript}

\usepackage{stackengine}
\usepackage[justification=centering]{caption}

\usepackage{pgfplots}

\pgfplotsset{width=10cm,compat=1.8}

\usepackage{pgfkeys}
    \newenvironment{customlegend}[1][]{%
        \begingroup
        \csname pgfplots@init@cleared@structures\endcsname
        \pgfplotsset{#1}%
    }{%
        \csname pgfplots@createlegend\endcsname
        \endgroup
    }%
    \def\addlegendimage{
    \csname pgfplots@addlegendimage\endcsname}

\ifthenelse{\equal{\mode}{JournalMode}}{
\usepackage[utf8]{inputenc}
}{
\usepackage[a4paper]{geometry}

\usepackage[hidelinks,hypertexnames=false,hyperindex=true,pdfpagelabels,linktoc=all]{hyperref}
\urlstyle{same}

\usepackage{amsfonts,amsmath,amssymb}
\usepackage{graphicx}

\usepackage[usenames,dvipsnames,svgnames]{xcolor}

\usepackage{amsthm}

\usepackage{array}
\usepackage{booktabs}

\usepackage{multirow}
}

\ifstrequal{\mode}{JournalMode}{
	\jyear{2021}
	}

\theoremstyle{thmstyleone}
\newtheorem{theorem}{Theorem}[section]
\newtheorem{lemma}[theorem]{Lemma}
\newtheorem{algo}[theorem]{Algorithm}

\theoremstyle{thmstyletwo}

\theoremstyle{thmstylethree}

\numberwithin{equation}{section}

\newcommand\xrowht[2][0]{\addstackgap[.5\dimexpr#2\relax]{\vphantom{#1}}}

\newcommand{\bsgamma}{{\boldsymbol{\gamma}}}

\newcommand{\bsh}{{\boldsymbol{h}}}
\newcommand{\bsp}{{\boldsymbol{p}}}
\newcommand{\bsell}{{\boldsymbol{\ell}}}
\newcommand{\bst}{{\boldsymbol{t}}}

\newcommand{\bsx}{{\boldsymbol{x}}}
\newcommand{\bsq}{{\boldsymbol{q}}}
\newcommand{\bsy}{{\boldsymbol{y}}}
\newcommand{\bsz}{{\boldsymbol{z}}}

\newcommand{\bszero}{{\boldsymbol{0}}}

\newcommand{\rd}{\mathrm{d}}

\newcommand{\rme}{\mathrm{e}}
\newcommand{\ri}{\mathrm{i}}

\newcommand{\bbZ}{\mathbb{Z}}
\newcommand{\bbN}{\mathbb{N}}

\newcommand{\bbU}{\mathbb{U}}
\newcommand{\calA}{\mathcal{A}}

\newcommand{\calO}{\mathcal{O}}

\newcommand{\scrP}{\mathscr{P}}
\newcommand{\scrQ}{\mathscr{Q}}

\newcommand\supp{{\mathrm{supp}}}
\newcommand{\setu}{{\mathfrak{u}}}

\newcommand{\setw}{{\mathfrak{w}}}

\definecolor{darkred}{RGB}{139,0,0}
\definecolor{darkgreen}{RGB}{0,100,0}
\definecolor{darkmagenta}{RGB}{170,0,120}
\definecolor{darkpurple}{RGB}{110,0,180}
\definecolor{darkblue}{RGB}{40,0,200}
\definecolor{darkbrown}{rgb}{0.75,0.40,0.15}

\newcommand{\wm}[1]{{\color{darkgreen}{#1}}}

\begin{document}

\ifthenelse{\equal{\mode}{JournalMode}}{
\title[Constructing Lattice Algorithms for Approximation with Composite $n$]
 { \centering Constructing \\Embedded Lattice-based Algorithms \\ for Multivariate Function Approximation \\
  with a Composite Number of Points}

\author[1]{\fnm{Frances Y.} \sur{Kuo}}\email{f.kuo@unsw.edu.au}

\author*[2]{\fnm{Weiwen} \sur{Mo}}\email{weiwen.mo@kuleuven.be}

\author[2]{\fnm{Dirk} \sur{Nuyens}}\email{dirk.nuyens@kuleuven.be}

\affil[1]{\orgdiv{School of Mathematics and Statistics},
           \orgname{UNSW Sydney},
           \orgaddress{\city{Sydney} \state{NSW} \postcode{2052}, \country{Australia}}}
           
\affil*[2]{\orgdiv{Department of Computer Science},
          \orgname{KU Leuven},
          \orgaddress{\street{Celestijnenlaan 200A}, \postcode{3001} \city{Leuven}, \country{Belgium}}}

\abstract{
\protect\unboldmath
We approximate $d$-variate periodic functions in weighted Korobov spaces
with general weight parameters using $n$ function values at lattice
points. We do not limit $n$ to be a prime number, as in currently
available literature, but allow any number of points, including powers of
$2$, thus providing the fundamental theory for construction of
embedded lattice sequences. Our results are constructive in that we provide a
component-by-component algorithm which constructs a suitable generating
vector for a given number of points or even a range of numbers of points. It does so without needing to
construct the index set on which the functions will be represented. The
resulting 
generating vector can then be used to approximate functions in the
underlying weighted Korobov space. We analyse the approximation error in
the worst-case setting under both the $L_2$ and $L_{\infty}$ norms. Our
component-by-component construction under the $L_2$ norm achieves the best
possible rate of convergence for lattice-based algorithms, and the theory
can be applied to lattice-based kernel methods and splines. Depending on
the value of the smoothness parameter $\alpha$, we propose two variants of
the search criterion in the construction under the $L_{\infty}$ norm,
extending previous results which hold only for product-type weight
parameters and prime $n$.
We also provide a theoretical upper bound showing that embedded lattice 
sequences are essentially as good as lattice rules with a fixed value of $n$.
Under some standard assumptions on the weight
parameters, the worst-case error bound is independent of~$d$.}



\maketitle

}
{
\title{Constructing 
\\
Embedded Lattice-based Algorithms 
\\ 
for Multivariate Function Approximation 
\\
with a Composite Number of Points}

\author{
    Frances Y.~Kuo\footnote{School of Mathematics and Statistics,
        University of New South Wales, Sydney NSW 2052, Australia,
        \texttt{f.kuo@unsw.edu.au}}
\and
	Weiwen Mo\footnote{Department of Computer Science, KU Leuven,
        Celestijnenlaan 200A, 3001 Leuven, Belgium,
        \texttt{(weiwen.mo|dirk.nuyens)@kuleuven.be}}
\and
    Dirk Nuyens \footnotemark[2]
    }
    
\date{August 2022}

\maketitle

\abstract{
We approximate $d$-variate periodic functions in weighted Korobov spaces
with general weight parameters using $n$ function values at lattice
points. We do not limit $n$ to be a prime number, as in currently
available literature, but allow any number of points, including powers of
$2$, thus providing the fundamental theory for construction of
embedded lattice sequences. Our results are constructive in that we provide a
component-by-component algorithm which constructs a suitable generating
vector for a given number of points or even a range of numbers of points. It does so without needing to
construct the index set on which the functions will be represented. The
resulting 
generating vector can then be used to approximate functions in the
underlying weighted Korobov space. We analyse the approximation error in
the worst-case setting under both the $L_2$ and $L_{\infty}$ norms. Our
component-by-component construction under the $L_2$ norm achieves the best
possible rate of convergence for lattice-based algorithms, and the theory
can be applied to lattice-based kernel methods and splines. Depending on
the value of the smoothness parameter $\alpha$, we propose two variants of
the search criterion in the construction under the $L_{\infty}$ norm,
extending previous results which hold only for product-type weight
parameters and prime $n$.
We also provide a theoretical upper bound showing that embedded lattice 
sequences are essentially as good as lattice rules with a fixed value of $n$.
Under some standard assumptions on the weight
parameters, the worst-case error bound is independent of~$d$.


}
}

\section{Introduction}
In this paper we provide a theoretical foundation for the
\emph{component-by-component} (CBC) \emph{construction} of \emph{lattice
algorithms} and a practical CBC construction of \emph{embedded lattice sequences} \cite{CKN06}
for multivariate $L_2$ and $L_{\infty}$ approximation in the
worst-case setting, for $d$-variate functions $f$ in weighted Korobov
spaces with smoothness parameter $\alpha$ and general weights
parameters~$\bsgamma := \{ \gamma_{\setu} \}_{\setu \subset \bbN }$ (see
Section~\ref{sec:Krob} for details). The algorithm is based on function
values at $n$ lattice points (see \eqref{eq:Af} below). Our error
analysis for the CBC construction allows for any $n \ge 2$. Currently
available literature on lattice-based algorithms provide error analysis restricted to prime
$n$ and, in the case of $L_\infty$ approximation, further restricted to
only ``product''-type weight parameters.

Although our work here might look quite theoretical, it
is motivated by strong practical needs. Composite values of $n$
enable practical applications of embedded lattice sequences, e.g., with $n$ being successive powers of $2$, while
non-product weight parameters are crucial for some uncertainty
quantification problems involving PDEs with random coefficients (see POD
and SPOD weights in \cite{KKKNS21} and further references below).
Furthermore, our
analysis also applies to \emph{lattice-based kernel methods},
see~\cite{KKKNS21}.

More precisely, we consider one-periodic real-valued $L_2$ functions
defined on $[0,1]^d$ with absolutely converging Fourier series
\begin{align*}
  f(\bsx) &\,=\, \sum_{\bsh\in\bbZ^d} \widehat{f} (\bsh) \, \rme^{2\pi\ri\bsh\cdot\bsx},
  &
  \widehat{f} (\bsh) &\,:=\, \int_{[0,1]^d} f(\bsx)\, \rme^{-2\pi\ri\bsh\cdot\bsx}\,\rd\bsx,
\end{align*}
where $\widehat{f}(\bsh)$ are the Fourier coefficients and $\bsh\cdot\bsx
= h_1 x_1 + \cdots + h_d x_d$ denotes the usual dot product. The norm of
our function space will be defined in terms of the Fourier coefficients
(see~\eqref{eq:norm} below), and when the smoothness parameter
$\alpha$ of the weighted Korobov space is even then this has the
interpretation that $f$ has square-integrable mixed partial derivatives of
order $\alpha/2$.

The lattice algorithm $A_n(f)$ is defined as follows: we first
truncate the Fourier expansion to a finite index set
$\calA_d\subset\bbZ^d$ and then approximate the remaining Fourier
coefficients by $n$ \emph{rank-$1$ lattice points}, i.e.,
\begin{align} \label{eq:Af}
\begin{split}
  A_n(f)(\bsx) := \sum_{\bsh\in\calA_d} \widehat{f}^a (\bsh) \,\rme^{2\pi\ri\bsh\cdot\bsx},
  \qquad
 \widehat{f}^a (\bsh) &:= \frac{1}{n} \sum_{k=1}^n f(\bst_k)\,\rme^{-2\pi\ri \bsh \wm{\cdot} \bst_k},
  \\
  \bst_k &:= \Big\{\frac{k\bsz}{n}\Big\},
\end{split}
\end{align}
where $\bsz \in \bbU_n^d$ is known as the \emph{generating vector}
which determines the quality of the approximation, with components
from
\begin{align*} 
    \bbU_n \,:=\,  \big\{ z \in \bbZ \,: \, 1 \leq z\leq n-1 \, \, \textnormal{and} \, \, \gcd(z,n)=1  \big\},
\end{align*}
and the braces in~\eqref{eq:Af} indicate that we take the fractional
part of each component of a vector.

In this paper we follow \cite{CKNS20,CKNS21,KSW06,KWW09c} to
define the index set $\calA_d$ with some parameter $M>0$ by
\begin{align} \label{eq:AdM}
  \calA_d(M) \,:=\, \big\{\bsh\in\bbZ^d : r_{d,\alpha,\bsgamma}(\bsh) \le M \big\},
\end{align}
where the quantity $r_{d,\alpha,\bsgamma}(\bsh)$ (see \eqref{eq:rda}
below) moderates the decay of $\vert\widehat{f} (\bsh)\vert$, with a smaller value
of $r_{d,\alpha,\bsgamma}(\bsh)$ indicating that the index $\bsh$ is
more significant. The set $\calA_d(M)$ therefore collects the most
significant indices up to a threshold~$M$.

Lattice rules were originally designed for multivariate integration, see,
e.g., \cite{CN08, DKS13, LM12, Lem09, LP14, Nie92, Nuy14, SJ94}, but the
benefits of lattice-based algorithms for multivariate function
approximation have also been recognized. A key development in recent years
is the \emph{CBC construction of lattice generating vectors} in high
dimensions, with guaranteed good theoretical error bounds for integration
and approximation in a variety of settings. Specifically for
approximation, one approach is based on \emph{reconstruction lattices} in
which the generating vector $\bsz$ is constructed to recover certain
Fourier coefficients exactly so that $\widehat{f}^a (\bsh) =
\widehat{f}(\bsh)$ for all $\bsh \in \calA_d$ in~\eqref{eq:Af}, where
$\calA_d$ can be any index set not restricted to the form~\eqref{eq:AdM},
see, e.g., \cite{BKUV17,Kam13,PV16,KPV15,KMNN}. Combining multiple
reconstruction lattices has also been shown to improve the error, see,
e.g., \cite{GIKV21, LK18, LK19, LK19b, KV19}. Another approach constructs
the generating vector $\bsz$ to minimize some search criterion, with the
aim to directly control the approximation error, see, e.g., \cite{KSW06,
KWW09c, CKNS20, CKNS21}. We follow the latter approach in this paper.

We measure the approximation error of the algorithm $A_n$ in the
weighted Korobov space in the worst-case setting under the $L_2$ and
$L_\infty$ norms:
\begin{align}\label{eq:def:wce}
 e^{\rm wor\mbox{-}app} (A_n;L_q)
 &\,:=\, \sup_{\|f\|_{d,\alpha,\bsgamma} \leq 1} \|f - A_n(f)\|_{L_q ([0,1]^d)},
 &
 q &\in \{2,\infty\},
\end{align}
where $\|f\|_{d,\alpha,\bsgamma}$ denotes our Korobov space norm (see
\eqref{eq:norm} below). Our goal is in constructing the lattice generating
vector $\bsz$ to have the \emph{largest possible rate of convergence} $r$
in $e^{\rm wor\mbox{-}app} (A_n;L_q) = \calO(n^{-r+\delta})$ for
arbitrarily small $\delta>0$, with the \emph{implied constant independent
of $d$} under appropriate conditions on the weight parameters~$\bsgamma$.
This is the concept of \emph{strong tractability}, see, e.g.,
\cite{NW08,NW10,NW12}.

Many papers study general multivariate approximation problems under
different assumptions on the available information. The class
$\Lambda^{\rm all}$ of \emph{arbitrary linear information} allows all
linear functionals, while the class $\Lambda^{\rm std}$ of \emph{standard
information} allows only function values. It is still an open problem
whether algorithms in $\Lambda^{\rm std}$ can achieve the same convergence
rate as $\Lambda^{\rm all}$. Our approximation~\eqref{eq:Af} involves
function values at lattice points, so it falls under the
class~$\Lambda^{\rm std}$.
Table~\ref{tab:conr} lists the optimal or best-known convergence rates in
the weighted Korobov space, comparing general results from $\Lambda^{\rm
all}$ and $\Lambda^{\rm std}$ to what can be achieved by lattice
algorithms. (It needs to be noted that different papers use different
conventions for the smoothness parameter $\alpha$ and therefore might
use $2\alpha$ where we use $\alpha$.) We clearly see that lattice
algorithms are not optimal. Indeed, it was proved in \cite{BKUV17} that
the best possible convergence rate for lattice-based algorithms is only
$\alpha/4$. However, lattice algorithms are easy and efficient to
construct and implement, compared to the general results from
$\Lambda^{\rm std}$ which are typically non-constructive, see, e.g.,
\cite{KU21, NSU21}.

\begin{table}
\caption{Comparison of convergence rates for approximation
algorithms in the weighted Korobov space with smoothness parameter
$\alpha>1$; in this paper we extend the lattice algorithm results to composite $n$ and general weights}\label{tab:conr}
\centering
\begin{tabular}[t]{ l l l l}
  \toprule
   & $\Lambda^{\rm all}$ & $\Lambda^{\rm std}$  & Lattice algorithms   \\
  \midrule 
  $L_2$  & $ \displaystyle\frac{\alpha}{2} $ \cite{NSW04}
                    & $\geq \displaystyle\frac{\alpha}{2} \frac{1}{1 + 1/\alpha} $ \cite{KWW09a}
                    & $\displaystyle\frac{\alpha}{4} $ \cite{BKUV17, KSW06, CKNS20}\\[2mm]
  $L_{\infty}$  &  $\displaystyle\frac{\alpha-1}{2} $  \cite{KWW08a} &  $\geq \displaystyle\frac{\alpha-1}{2} \frac{1}{1 + 1/\alpha} $ \cite{KWW09a} & $\geq \begin{cases}
  \displaystyle\frac{\alpha -1}{4}  & \mbox{ for } \alpha > 1 \\[3mm]
  \displaystyle\frac{\alpha -1}{2} \displaystyle\frac{1}{2 - 1/\alpha} & \mbox{ for } \alpha > 2
  \end{cases} $ \cite{KWW09c} \\
  \bottomrule
\end{tabular}
\end{table}

The lattice approximation~\eqref{eq:Af} with prime $n$ was already analyzed in
the worst-case $L_2$ and $L_{\infty}$ settings in the weighted Korobov
space with product weights, see \cite{KSW06} and \cite{KWW09c},
respectively, and the last column of Table~\ref{tab:conr}. Motivated by
the need from PDE applications \cite{KKKNS21}, theoretical justification
for the CBC construction with general weights was developed in the $L_2$
setting in \cite{CKNS20}, requiring a substantially more complicated
analysis due to the difficulty of handling non-product weights in an
inductive argument. Subsequently, fast CBC implementations were developed
for special forms of weights (order-dependent, POD and SPOD weights) in
\cite{CKNS21}, again requiring substantially more complex computational
strategies, including the use of both fast Fourier and fast Hankel
transforms. While the fast CBC implementations from \cite{CKNS21} are
applicable to composite $n$, the theoretical justification in
\cite{CKNS20} is restricted to prime $n$. The current paper removes this
restriction by employing a different proof technique (see further below).

We use the same search criterion $S_{n,d,\alpha, \bsgamma}(\bsz)$
from \cite{CKNS20,CKNS21} (see \eqref{eq:Sd} below). Although the
approximation~\eqref{eq:Af} depends on the index set $\calA_d(M)$, the
quantity $S_{n,d,\alpha, \bsgamma}(\bsz)$ is independent of the index
set and the value of $M$. This is advantageous for the computational cost
since there is no need to create and maintain the index set during the CBC
construction. It was shown, e.g., in \cite{CKNS20} that the worst-case
$L_2$ approximation error for our lattice approximation~\eqref{eq:Af}
satisfies
\begin{align} \label{ieq:errMS}
  e^{\rm wor\mbox{-}app}(A_n;L_{2})  \,\le\, \bigg(\frac{1}{M} + M\, S_{n,d, \alpha,\bsgamma}(\bsz) \bigg)^{1/2}.
\end{align}
We show in Section~\ref{sec:WceLinf} that the worst-case $L_{\infty}$
approximation error satisfies
\begin{align*}
 & e^{\rm wor\mbox{-}app}(A_n;L_{\infty})  \\
& \leq
 \begin{cases}
  	\displaystyle \bigg( \sum_{\bsh \not\in \calA_d(M) } \frac{1}{r_{d,\alpha,\bsgamma}(\bsh)} +  3\, M  \left\lvert \calA_d(M) \right\rvert  S_{n,d, \alpha, \bsgamma}(\bsz )  \bigg)^{1/2}
		&  \mbox{if } \alpha > 1, \\
        \displaystyle \bigg( \sum_{\bsh \not\in \calA_d(M) } \frac{1}{r_{d,\alpha,\bsgamma}(\bsh)} +  3\, M \,
        \big[S_{n,d, \frac{\alpha}{2}, \sqrt{\bsgamma}}(\bsz)\big]^2 \bigg)^{1/2}
 		&  \mbox{if } \alpha > 2.
\end{cases}
\end{align*}
These error bounds hold for general weights and for prime and
composite~$n$. Note that each bound involves a sum of two terms,
representing the truncation error and the cubature error. We choose $M$ to
balance the two terms.

We establish in Theorem~\textup{\ref{thm:SUapp}} that the CBC construction
with a general $n$ achieves essentially the same bound on the quantity
$S_{n,d,\alpha,\bsgamma}(\bsz)$ as for prime $n$ in~\cite{CKNS20},
thus completing the theory for $L_2$ approximation. For $L_\infty$
approximation, the same CBC construction and bound on
$S_{n,d,\alpha,\bsgamma}(\bsz)$ can be applied when $\alpha>1$, while
for a higher smoothness $\alpha>2$ we can alternatively carry out the CBC
construction with $\alpha$ replaced by $\alpha/2$ and all weights
$\gamma_\setu$ replaced by $\sqrt{\gamma_\setu}$, and then revise the
bound accordingly. This yields the same convergence rates for $L_\infty$
approximation as in \cite{KWW09c}, thus extending the final entry in
Table~\ref{tab:conr} to general weights and composite $n$. Recall that the
fast implementations for special forms of weights in \cite{CKNS21} are
already applicable to composite~$n$.

The $L_2$ approximation result from this paper serves as an immediate
and crucial upper bound for lattice-based kernel methods \cite{KKKNS21,
Wah90, ZLH06, ZKH09}, defined by
\begin{align*}
  A^{\rm ker}_n(f)(\bsx) \,:=\, \sum_{k=1}^n a_k\, K(\bsx,\bst_k),
\end{align*}
where $\bst_k$ are lattice points as in~\eqref{eq:Af}, $K(\bsx,\bsy)$ is
the reproducing kernel of the weighted Korobov space (see \eqref{eq:ker}
below), and the coefficients $a_k$ are such that $A^{\rm ker}_n(f)$
interpolates $f$ at the $n$ lattice points $\bst_k$. These coefficients
$a_k$ can be found by solving a linear system which has a circulant
matrix, so they can be obtained efficiently using the fast Fourier
transform. As explained in \cite{KKKNS21}, kernel methods are optimal for
$L_q$ approximation among all algorithms that use the same points. Thus
with the same lattice generating vector $\bsz$ we have
\begin{align*}
  e^{\rm wor\mbox{-}app}(A^{\rm ker}_n;L_q) \,\le\, e^{\rm wor\mbox{-}app}(A_n;L_q),
  \quad 1 \le q \le \infty,
\end{align*}
which is why our error bound and CBC construction can be applied directly.
Recall that our approximation $A_n(f)$ in~\eqref{eq:Af} depends on
the index set $\calA_d(M)$, whereas the kernel method $A^{\rm
ker}_n(f)$ involves no index set. Therefore, the kernel method can
completely by-pass the index set, since our search criterion $S_{n,d,
\alpha,\bsgamma}(\bsz)$ is also independent of the index set.

It is important to note that the new results in this paper cannot be
obtained by trivial generalisations of existing results for prime~$n$.
The analysis for a lattice generating vector $(\bsz,z_s)\in\bbU_n^s$
typically involves a certain sum over vectors $(\bsell,\ell_s)\in\bbZ^s$
satisfying the congruence
\[
   \bsell\cdot \bsz + \ell_s z_s \equiv 0 \pmod n.
\]
Existing averaging techniques for numerical integration with non-prime $n$
work by first splitting the sum based on whether or not $\ell_s$ is a
multiple of $n$ and then counting the repeated values of $\bsell\cdot\bsz
\bmod n$, see, e.g., \cite[proof of Theorem~5.8]{DKS13}. An attempt to use
that technique for the approximation problem with non-prime $n$ led to an
undesirable error bound. Instead, in this paper we change the splitting to
be based on the values of $\gcd(\ell_s,n)$ together with
$\gcd(\bsell\cdot\bsz,n)$, see Lemma~\textup{\ref{lem:numg}} and
Lemma~\textup{\ref{lem:thetau}} ahead. This new proof technique can also
be used in the context of integration, leading to the same result there
but with a shorter proof.

Embedded or extensible lattice sequences for integration have been analysed in the past, 
see, e.g., \cite{CKN06, DPW08, HHLL00, HN03}.
Here we follow a similar approach to \cite{CKN06} where we 
apply a mini-max strategy based on the ratios of  a dimension-wise decomposition of 
$S_{n,d,\alpha,\bsgamma}(\bsz)$ to construct the embedded lattice sequences for approximation (see Algorithm~\ref{alg:CBCemb} below)
for $n$ being successive powers of a prime, i.e., $n = p^m$ for $m_1 \le m \le m_2$.
With $N = p^{m_2}$, Theorem~\ref{thm:upbX} proves that the embedded sequence is only a factor of $(\log N)^{\alpha}$
worse than lattice rules with a fixed number of points.
Numerical results in Section~\textup{\ref{sec:NumRes}} confirms this in practice.

The paper is organised as follows. Section~\textup{\ref{sec:Krob}} introduces the
weighted Korobov spaces. Section~\textup{\ref{sec:WceL2}} provides a
brief review of known results for $L_2$ approximation and states the
new error bound for general~$n$. Section~\textup{\ref{sec:cbc}} is devoted to
the technical proof of the main theorem. Section~\textup{\ref{sec:WceLinf}}
derives an upper bound on the worst-case error in $L_{\infty}$
norm and the corresponding search criterion for CBC construction.
Section~\textup{\ref{sec:Emb_lat}} constructs good
generating vectors of embedded lattice sequences for a range of number of
points.
Section~\textup{\ref{sec:NumRes}} presents numerical results which
confirm the convergence rates of our search criteria.

\section{Weighted Korobov spaces} \label{sec:Krob}

For $\alpha>1$ and positive weight parameters $\bsgamma
:=\{\gamma_\setu\}_{\setu \subset \bbN}$, we consider the Hilbert space
$H_d$ of one-periodic $L_2$ functions defined on $[0,1]^d$
with absolutely convergent Fourier series, with norm defined by
\begin{align} \label{eq:norm}
  \|f\|_{d,\alpha, \bsgamma}^2
  \,:=\, \sum_{\bsh\in\bbZ^d} \big\lvert \widehat{f} (\bsh) \big\rvert ^2\,r_{d,\alpha,\bsgamma}(\bsh) ,
\end{align}
where
\begin{align} \label{eq:rda}
 r_{d,\alpha,\bsgamma}(\bsh) \,:=\, \frac{1}{\gamma_{\supp(\bsh)}}\,
 \prod_{j\in \supp(\bsh)} \lvert h_j \rvert ^{\alpha},
\end{align}
with $\supp(\bsh) := \{1\le j\le d : h_j \ne 0\}$. The parameter $\alpha$
characterizes the rate of decay of the Fourier coefficients in the norm, so it
is a smoothness parameter. We fix the scaling of the weights by setting
$\gamma_\emptyset := 1$, so that the norm of a constant function in $H_d$
matches its $L_2$ norm and  $L_{\infty}$ norm. Various forms of weight
parameters appear in the literature for multivariate integration, such as
product weights \cite{SW98, SW01}, order-dependent weights \cite{DSWW06},
POD weights \cite{GKNSSS15, GKNSS18b, KN16, KSS12} and SPOD weights
\cite{DKLS16, KKS20}. In the simplest case of product weights, there is
one weight $\gamma_j$ for each coordinate index $j$, and the weight for a
set $\setu$ of coordinate indices is given by $\gamma_\setu =
\prod_{j\in\setu} \gamma_j$. The other types of weights are more
complicated and are motivated by applications, see the references above.

Some authors refer to this space as the \emph{weighted Korobov space}, see
\cite{SW01} for product weights and \cite{DSWW06} for general weights,
while others call this a weighted variant of the \emph{periodic Sobolev
space with dominating mixed smoothness}~\cite{BKUV17}.
We remark again that some references might use $2\alpha$ where we use $\alpha$
and also sometimes the weight parameters might occur squared in the Hilbert norm.

When $\alpha\geq 2$ is an even integer, it can be shown that
\begin{align*}
 \|f\|_{d,\alpha,\bsgamma}^2 \! = \! \!
 \sum_{\setu\subseteq\{1:d\}} \! \! \frac{1}{(2\pi)^{\alpha \lvert\setu\rvert}}\frac{1}{\gamma_{\setu}}
 \int_{[0,1]^{\lvert \setu \rvert}} \!\!\!
 \bigg(\int_{[0,1]^{d-\lvert\setu\rvert}} \!\!\! \bigg(\prod_{j\in\setu}\frac{\partial}{\partial x_j}\bigg)^{\!\!\frac{\alpha}{2}}f(\bsx)
 \,\rd\bsx_{\{1:d\}\setminus\setu}\bigg)^2 \!
 \rd\bsx_{\setu}.
\end{align*}
So $f$ has square-integrable mixed partial derivatives of order $\alpha/2$
over all possible subsets of variables. Here and elsewhere in the
paper, $\{1\mathbin{:}d\} = \{1,2,\ldots,d\}$ and $\bsx_\setu =
(x_j)_{j\in\setu}$.

The inner product of $H_d$ is given by
\begin{align}\label{eq:ip}
  \left \langle f,g \right \rangle_{d, \alpha,\bsgamma}
  \, := \,
  \sum_{\bsh \in \bbZ^d} \widehat{f} (\bsh) \, \overline{\widehat{g} (\bsh)} \, r_{d,\alpha,\bsgamma}(\bsh),
\end{align}
and the norm is $\| \cdot \|_{d, \alpha,\bsgamma} = \langle \cdot,\cdot
\rangle_{d, \alpha,\bsgamma}^{1/2}$ which is consistent with
\eqref{eq:norm}. The reproducing kernel for $H_d$ is
\begin{align} \label{eq:ker}
K_d(\bsx,\bsy) = \sum_{\bsh \in \bbZ^d}
				\frac{ \rme^{2 \pi \rm{i} \bsh \cdot (\bsx - \bsy)} }{ r_{d,\alpha,\bsgamma}(\bsh)  },
\end{align}
which satisfies (i) $K_d(\bsx, \bsy) = K_d(\bsy, \bsx)$ for all
$\bsx, \bsy \in [0,1]^d$; (ii) $K_d(\cdot, \bsy) \in H_d$ for all
$\bsy \in [0,1]^d$; (iii) $\left \langle f , K_d(\cdot, \bsy) \right
\rangle_d = f(\bsy)$ for all $f \in H_d$ and all $\bsy \in [0,1]^d$.
The last property is known as the \emph{reproducing property}. Note
that $r_{d,\alpha,\bsgamma}(\bsh) =r_{d,\alpha, \bsgamma} (-\bsh) $ and
therefore $K_d(\cdot,\cdot)$ takes real values and can be written as a sum
of cosine functions.

To simplify our notation, from this point on we write
$$
r(\bsh) \, := \, r_{d,\alpha,\bsgamma}(\bsh),
$$
except when we need to show the explicit dependence on $d$, $\alpha$
and $\bsgamma$.

\section{Worst-case $L_{2}$ error with general $n$} \label{sec:WceL2}

The error for the lattice approximation in~\eqref{eq:Af} clearly splits
into two terms, i.e., the truncation error and the cubature error,
\begin{align} \label{eq:aperr}
 (f-A_n(f))(\bsx)
 \,=\, \sum_{\bsh \not\in \calA_{d}(M)}  \widehat{f} (\bsh) \, \rme^{2 \pi \ri \bsh \cdot {\bsx}}
 + \sum_{\bsh \in \calA_{d}(M)} \Big( \widehat{f} (\bsh) - \widehat{f}^a (\bsh) \Big)\,  \rme^{2 \pi \ri \bsh \cdot {\bsx}}.
\end{align}
When measured in the $L_2$ norm, we have
\begin{align*}
\|f - A_n(f) \|^2_{L_2([0,1]^d)} = \sum_{\bsh \not\in \calA_{d}(M)} \lvert  \widehat{f} (\bsh) \rvert^2 + \sum_{\bsh \in \calA_{d}(M)} \lvert\widehat{f} (\bsh) -  \widehat{f}^a (\bsh) \rvert^2.
\end{align*}
From \cite{CKNS20} we have the following bound on the worst-case $L_2$
approximation error~\eqref{eq:def:wce} where we now denote explicitly the
dependence on the generating vector $\bsz$ and the value of $M$ in the notation
\begin{align} \label{eq:bal}
  e^{\rm wor\mbox{-}app}_{n,d,M} (\bsz;L_{2})
  \,\le\, \bigg(\frac{1}{M} + E_{n,d, \alpha, \bsgamma}(\bsz)\bigg)^{1/2}
  \,\le\, \bigg(\frac{1}{M} + M\, S_{n,d, \alpha,\bsgamma}(\bsz ) \bigg)^{1/2},
\end{align}
with
\begin{align} \label{eq:Sd}
  E_{n,d, \alpha, \bsgamma}(\bsz)
  &\,:=\, \sum_{\bsh\in\calA_d(M)}
  \sum_{\substack{\bsell\in\bbZ^d\setminus\{\bszero\} \\ \bsell\cdot\bsz\equiv_n 0}}
  \frac{1}{r_{d,\alpha, \bsgamma} (\bsh + \bsell)}
  ,
  \nonumber \\
  S_{n,d,\alpha, \bsgamma}(\bsz)
  &\,:=\, \sum_{\bsh\in\bbZ^d} \frac{1}{r_{d,\alpha,\bsgamma}(\bsh)}
  \sum_{\substack{\bsell\in\bbZ^d\setminus\{\bszero\} \\ \bsell\cdot\bsz\equiv_n 0}}
  \frac{1}{r_{d,\alpha, \bsgamma} (\bsh + \bsell)}
  ,
\end{align}
where $\equiv_n$ means congruence modulo $n$.

The advantage of the search criterion $S_{n,d, \alpha, \bsgamma}(\bsz
)$ from \cite{CKNS20} over $ E_{n,d, \alpha, \bsgamma}(\bsz ) $ from
\cite{KSW06} is that there is no dependence on the index set $\calA_d(M)$,
thus the error analysis is simpler and the construction cost is lower.

To handle general weights, instead of working directly with the error
criterion $ S_{n,d, \alpha, \bsgamma}(\bsz ) $, the CBC construction
in \cite{CKNS20} works with a \emph{dimension-wise decomposition} of $
S_{n,d, \alpha, \bsgamma}(\bsz ) $, see
Lemma~\textup{\ref{lem:Sdecomp}} and Algorithm~\textup{\ref{alg:CBCapp}}
below, both are valid for general $n$. Lemma~\textup{\ref{lem:Sdecomp}}
can be shown by a recursive decomposition as in \cite[Lemma~3.2]{CKNS20}
or it can be proved directly by induction.

\begin{lemma} \label{lem:Sdecomp}
Given $n \geq 2$ \textnormal{(}prime or composite\textnormal{)}, $d \geq
1$, $\alpha>1$, and weights $\bsgamma=\{\gamma_\setu\}_{\setu\subset
\bbN}$, we can write
\begin{align*} 
  S_{n,d, \alpha, \bsgamma}(\bsz )
  \,=\, \sum_{s=1}^d T_{n,d,s,\alpha,\bsgamma} \big(z_1,\ldots,z_s\big),
\end{align*}
where, for each $s=1,2,\ldots,d$,
\begin{align} \label{eq:Tds}
  T_{n,d,s,\alpha,\bsgamma} \big(z_1,\ldots,z_s\big)
  \,:=\, \sum_{\setw\subseteq\{s+1:d\}}
  [2\zeta(2\alpha)]^{\lvert\setw\rvert}\, \theta_{n,s,\alpha} \big(z_1,\ldots,z_s;\{\gamma_{\setu\cup\setw}\}_{\setu\subseteq\{1:s\}}\big),
\end{align}
\begin{align} \label{eq:theta}
  \theta_{n,s,\alpha} \big(z_1,\ldots,z_s;\{\beta_\setu\}_{\setu\subseteq\{1:s\}}\big)
  \,:=\, \sum_{\bsh\in\bbZ^s} \sum_{\substack{\bsell\in\bbZ^s,\;\ell_s\ne 0 \\ \bsell\cdot (z_1,\ldots,z_s)\equiv_n 0}}
  \frac{\beta_{\supp(\bsh)}}{r'(\bsh)}
  \frac{\beta_{\supp(\bsh+\bsell)}}{r'(\bsh+\bsell)},
\end{align}
with $r'(\bsh) := \prod_{j\in\supp(\bsh)} \vert h_j\vert^\alpha$.
\end{lemma}

We remark that the expression \eqref{eq:theta} takes as input argument
a sequence of numbers $\beta_\setu$ indexed by $\setu \subseteq
\{1\mathbin{:}s\}$. The notation
$\{\gamma_{\setu\cup\setw}\}_{\setu\subseteq\{1:s\}}$ in~\eqref{eq:Tds}
indicates that we take different input arguments $\beta_\setu =
\gamma_{\setu\cup\setw}$ by varying $\setw$.

\begin{algo} \label{alg:CBCapp}
Given $n \geq 2$ \textnormal{(}prime or composite\textnormal{)}, $d \geq
1$, $\alpha>1$, and weights $\{\gamma_\setu\}_{\setu\subset \bbN}$,
 construct the generating vector $\bsz^*
= (z_1^*, \ldots, z_d^*)$ as follows: for each $s = 1,
\ldots,d$, with $z_1^*, \ldots, z_{s-1}^*$ fixed, choose $z_s$ from
$$\bbU_n = \left \{ z \in \bbZ \,: \, 1 \leq z\leq n-1 \, \, \textnormal{and} \, \, \gcd(z,n)=1  \right\} $$
to minimize $T_{n,d,s,\alpha,\bsgamma}(z_1^*, \ldots, z_{s-1}^*,
z_s)$ defined in~\eqref{eq:Tds}.
\end{algo}

Theorem~\textup{\ref{thm:SUapp}} below is a non-trivial extension of
\cite[Theorem~3.5]{CKNS20} to allow for composite $n$, showing an upper
bound on the search criterion $ S_{n,d, \alpha, \bsgamma}(\bsz ) $
guaranteed by the CBC construction. We devote
Section~\textup{\ref{sec:cbc}} to its technical proof.
Theorem~\textup{\ref{thm:errL2}} below is then a direct consequence of
\eqref{eq:bal} and Theorem~\textup{\ref{thm:SUapp}}, by setting $M$ to
satisfy $\frac{1}{M} =  M\, S_{n,d, \alpha, \bsgamma}(\bsz)$, showing
that the CBC construction for composite $n$ also achieves the best
possible convergence rate for $L_2$ approximation using lattice-based
algorithms as concluded in \cite{BKUV17}. Theorem~\textup{\ref{thm:errL2}}
is therefore a non-trivial extension of \cite[Theorem~3.6]{CKNS20} to
general $n$.

\begin{theorem} \label{thm:SUapp}
Given $n \geq 2$ \textnormal{(}prime or composite\textnormal{)}, $d \geq
1$, $\alpha>1$, and weights $\bsgamma=\{\gamma_\setu\}_{\setu\subset
\bbN}$, the generating vector $\bsz$ obtained from the CBC
construction following Algorithm~\textup{\ref{alg:CBCapp}}, satisfies for
all $\lambda \in (\frac{1}{\alpha},1]$,
\begin{align} \label{ieq:SUapp}
S_{n,d, \alpha, \bsgamma}(\bsz )
\, \leq \, \Bigg[
			\frac{\kappa}{\varphi(n)}
		\bigg( \sum_{ \emptyset \neq \setu \subseteq \{1:d\}  } \lvert \setu \rvert\, \gamma_{\setu}^{\lambda}
       [2 \zeta(\alpha \lambda)]^{\lvert \setu \rvert} \bigg)
		\bigg( \sum_{ \setu \subseteq \{1:d\}  } \gamma_{\setu}^{\lambda}
       [2 \zeta(\alpha \lambda)]^{\lvert \setu \rvert}   \bigg)
			\Bigg]^{1/\lambda},
\end{align}
where $\kappa:= 2^{2\alpha \lambda + 1} +1 $, $\zeta(x)$ denotes the
Riemann zeta function, i.e., $\zeta(x) := \sum_{h=1}^{\infty}
\frac{1}{h^x}$, and $\varphi(n)$ is the cardinality of $\bbU_n$, i.e., the
Euler's totient function.
\end{theorem}

\begin{theorem} \label{thm:errL2}
Given $n \geq 2$ \textnormal{(}prime or composite\textnormal{)}, $d \geq
1$, $\alpha>1$, and weights $\bsgamma=\{\gamma_\setu\}_{\setu\subset
\bbN}$, the lattice approximation~\eqref{eq:Af}, with index set
\eqref{eq:AdM} and generating vector $\bsz$ obtained from the CBC
construction following Algorithm~\textup{\ref{alg:CBCapp}}, after taking
$M$ in \eqref{eq:AdM} to be
$$ M = \left( S_{n,d, \alpha, \bsgamma}(\bsz ) \right)^{-1/2} ,$$
satisfies for all $\lambda\in (\frac{1}{\alpha},1]$, 
\begin{align*}
   e^{\rm wor\mbox{-}app}_{n,d,M} (\bsz;L_{2})
  &\,\le\, \sqrt{2} \left[ S_{n,d, \alpha, \bsgamma}(\bsz )   \right]^{1/4} \\
  & \, \le \, \frac{\sqrt{2}\,\kappa^{1/(4\lambda)}}{\varphi(n)^{1/(4\lambda)}} \bigg(
  \sum_{\setu\subseteq\{1:d\}} \max(\lvert\setu\rvert,1)\,\gamma_{\setu}^\lambda\,
  [2\zeta(\alpha\lambda)]^{\lvert\setu\rvert}\bigg)^{1/2\lambda},
\end{align*}
where $\kappa$, $\zeta(\cdot)$ and $\varphi(\cdot)$ are as in
Theorem~\textup{\ref{thm:SUapp}}. The exponent of $\varphi(n)^{-1}$ can be
arbitrarily close to
\[
  \frac{\alpha}{4}.
\]
The constant is independent of $d$ if
\begin{align} \label{ieq:idpcon}
 \sum_{\setu\subset\bbN,\,\lvert\setu\rvert<\infty}  \gamma_{\setu}^{\frac{1}{\alpha-4\delta}}\,
  [2 \, \rme^{1/\rme} \, \zeta\big(\tfrac{\alpha}{\alpha-4\delta}\big)]^{\lvert\setu\rvert}
  \,<\, \infty, \quad \delta >0.
\end{align}
\end{theorem}

The last part of the theorem used $\max(\lvert\setu\rvert,1) \leq (
\rme^{1/\rme})^{\lvert\setu\rvert} = (1.4446\ldots)^{\lvert\setu\rvert}$.

For $n = p^m$ being a power of a prime $p$, we have $1/\varphi(n) =
[1+1/(p-1)]/n \le 2/n$. For general $n$, it can be verified that
$1/\varphi(n) < 9/n$ for all $n<10^{30}$. Hence from the practical point
of view, $1/\varphi(n)$ can be considered as a constant factor times
$1/n$.

\section{Proof of Theorem~\textup{\ref{thm:SUapp}}} \label{sec:cbc}

Lemma~\textup{\ref{lem:thetau}} below provides the essential averaging
argument required in the proof of Theorem~\textup{\ref{thm:SUapp}}. We
begin by stating Lemma~\textup{\ref{lem:numg}} which is essential to
the proof of Lemma~\textup{\ref{lem:thetau}}. The proof of the first part
of Lemma~\textup{\ref{lem:numg}} can be found in a general textbook on
number theory, e.g., \cite[Chapters~5]{A76}. Based on the first part, the
second part of Lemma~\textup{\ref{lem:numg}} can be easily proved by using
the fact that the solutions are restricted to be coprime with $n$,
see, e.g., \cite[Lemma~5.5]{ACP10} or \cite[Theorem~2]{GP13}.

\begin{lemma} \label{lem:numg}
For any positive integer $n$, and given $\ell_s \in \bbZ$ and $c \in \bbZ$,
define $g = \gcd(\ell_s, n)$ and consider the linear congruence
\begin{align} \label{eq:cong1}
  c + \ell_s \, z_s \equiv_n 0 .
\end{align}
Then there are exactly $g$ solutions for $z_s \in \bbZ_n$ if $g \mid c$
and no solution if $g \nmid c$. Moreover, if we restrict the solutions to
$z_s \in \bbU_n$, i.e., $z_s \in \bbZ_n$ and $\gcd(z_s, n) = 1$, then there are
at most $g$ solutions
if also $\gcd(c, n) = g$ holds, with the same definition of $g = \gcd(\ell_s, n)$, and
no solutions otherwise.
\end{lemma}

\begin{lemma} \label{lem:thetau}
Given $n \geq 2$ \textnormal{(}prime or composite\textnormal{)}, for any
$s \geq 1$, any input sequence $\{\beta_{\setu}\}_{\setu \subseteq
\{1:s\}}$, all $(z_1, \ldots, z_{s-1}) \in \bbU_n^{s-1}$, and all $\lambda
\in (\frac{1}{\alpha},1]$, we have
\begin{align*}
&\frac{1}{\varphi(n)} \sum_{z_s \in \bbU_n} \left[ \theta_{n,s,\alpha} (z_1, \ldots,z_s; \{ \beta_{\setu  }\}_{ \setu \subseteq \{1:s\} } ) \right]^{\lambda} \notag \\
& \qquad \leq \frac{\kappa}{\varphi(n)}
\Bigg( \sum_{s \in \setu \subseteq \{1:s\}} \beta_{\setu}^{\lambda} [2 \zeta(\alpha\lambda)]^{\lvert\setu\rvert} \Bigg)
\Bigg( \sum_{\setu \subseteq \{1:s\}} \beta_{\setu}^{\lambda} [2 \zeta(\alpha\lambda)]^{\lvert\setu\rvert} \Bigg),
\end{align*}
where $\kappa$, $\zeta(\cdot)$ and $\varphi(\cdot)$ are as in
Theorem~\textup{\ref{thm:SUapp}}.
\end{lemma}

\begin{proof}
For notational convenience we write in this proof $\bsz :=
(z_1,\ldots,z_{s-1})\in\bbU_n^{s-1}$. From \eqref{eq:theta}, with a slight
change of notation for the indices $(\bsh,h_s)\in\bbZ^s$ and
$(\bsell,\ell_s)\in\bbZ^s$, and the inequality $\sum_k a_k \leq ( \sum_k
a_k^{\lambda} )^{1/\lambda}$ for all $a_k \geq 0$ and $\lambda \in
(1/\alpha,1]$, we obtain
\begin{align*}
&{\rm{Avg}}
:= \frac{1}{\varphi(n)} \sum_{z_s \in \bbU_n} \left[ \theta_{n,s,\alpha} (z_1, \ldots,z_s; \{ \beta_{\setu  }\}_{ \setu \subseteq \{1:s\} } ) \right]^{\lambda}\\
& \hphantom{:} \leq  \frac{1}{\varphi(n)} \sum_{z_s \in \bbU_n}
\sum_{(\bsh,h_s) \in \bbZ^s} \sum_{ \substack{ (\bsell,\ell_s) \in \bbZ^s, \ell_s \neq 0 \\ \bsell \cdot \bsz + \ell_s z_s  \equiv_n 0}}
 \frac{\beta^{\lambda}_{ \supp(\bsh , h_s)} }{r'(\bsh , h_s)^{\lambda}}
  \frac{\beta^{\lambda}_{ \supp((\bsh, h_s) + (\bsell,\ell_s))} }{r'((\bsh,h_s) + (\bsell,\ell_s))^{\lambda}} \\
 &\hphantom{:} =  \frac{1}{\varphi(n)}
 \sum_{(\bsh,h_s) \in \bbZ^s} \sum_{(\bsell,\ell_s) \in \bbZ^s, \ell_s \neq 0}
  \frac{\beta^{\lambda}_{ \supp(\bsh , h_s)} }{r'(\bsh , h_s)^{\lambda}}
  \frac{\beta^{\lambda}_{ \supp((\bsh, h_s) + (\bsell,\ell_s))} }{r'((\bsh,h_s) + (\bsell,\ell_s))^{\lambda}}
  \sum_{\substack{z_s \in \bbU_n \\ \bsell \cdot \bsz + \ell_s z_s  \equiv_n 0 }}  1 \\
  &\hphantom{:} \leq  \frac{1}{\varphi(n)} \sum_{\substack{ g \in \{1:n\} \\ g \mid n} } g
   \sum_{(\bsh,h_s) \in \bbZ^s}
   \sum_{\substack{(\bsell,\ell_s) \in \bbZ^s, \ell_s \neq 0 \\ \gcd(\bsell \cdot \bsz,n) = g \\ \gcd(\ell_s,n)=g}}
     \frac{\beta^{\lambda}_{ \supp(\bsh , h_s)} }{r'(\bsh , h_s)^{\lambda}}
  \frac{\beta^{\lambda}_{ \supp((\bsh, h_s) + (\bsell,\ell_s))} }{r'((\bsh,h_s) + (\bsell,\ell_s))^{\lambda}} \\
  &\hphantom{:} =  \frac{1}{\varphi(n)} \sum_{\substack{g \in \{1:n\} \\ g \mid n}} g
  	\sum_{h_s \in \bbZ}  \sum_{\substack{\ell_s \in \bbZ \setminus \{0\} \\\gcd(\ell_s,n)=g}}  G^{g}(h_s,\ell_s)
\end{align*}
with
\begin{align*}
G^{g}(h_s,\ell_s) := \sum_{\bsh \in \bbZ^{s-1}} \sum_{\substack{\bsell \in \bbZ^{s-1} \\ \gcd(\bsell \cdot \bsz,n) =g}}
	  \frac{\beta^{\lambda}_{ \supp(\bsh , h_s)} }{r'(\bsh , h_s)^{\lambda}}
  \frac{\beta^{\lambda}_{ \supp((\bsh, h_s) + (\bsell,\ell_s))} }{r'((\bsh,h_s) + (\bsell,\ell_s))^{\lambda}},
\end{align*}
where $r'(\bsh) = \prod_{j\in\supp(\bsh)} \lvert h_j \rvert ^\alpha$ as
defined in Lemma~\ref{lem:Sdecomp}, and the inequality holds because of
Lemma~\ref{lem:numg}, as we used that the congruence has at most $g$
solutions for $z_s \in \bbU_n$ if $g = \gcd(\ell_s, n) =
\gcd(\bsell\cdot\bsz, n)$ and no solution if this condition is not
satisfied. This step on splitting according to the divisors $g$ of $n$
is the key difference compared to the previous proof technique.

We separate the above expression into three parts: (i) $\ell_s = -h_s$;
(ii) $\ell_s \neq -h_s$ and $g \mid h_s$; (iii) $\ell_s \neq -h_s$ and $g
\nmid h_s$, to obtain
\begin{align*}
  &{\rm{Avg}}
  \leq
  \frac{1}{\varphi(n)}
  \Bigg(
\underbrace{ \sum_{\substack{g \in \{1:n\} \\ g \mid n}} g
	\sum_{\substack{h_s \in \bbZ \setminus \{0\} \\ \gcd(h_s,n)=g}}
	 \!\!\!\!\! G^{g}(h_s,-h_s) }_{=:\, F_1}
  \\
  &
  \qquad+ \underbrace{
    \sum_{\substack{g \in \{1:n\} \\ g \mid n}} g
  	\sum_{\substack{h_s \in \bbZ \\ g \mid h_s}}
	\;
	\sum_{\substack{\ell_s \in \bbZ \setminus \{0, -h_s\} \\ \gcd(\ell_s,n)=g}}
	 \!\!\!\!\! G^{g}(h_s,\ell_s) }_{=:\,F_2}
	+     \underbrace{
	\sum_{\substack{g \in \{1:n\} \\ g \mid n}} g
  	\sum_{\substack{h_s \in \bbZ \\ g \nmid h_s}}
	\;
	\sum_{\substack{\ell_s \in \bbZ \setminus \{0, -h_s\} \\ \gcd(\ell_s,n)=g}}
	 \!\!\!\!\! G^{g}(h_s,\ell_s) }_{=:\,F_3}
  \Bigg)
  .
\end{align*}

In the following we will swap the order of the multiple sums. For
$\ell_s \neq 0$, and with a relabeling of $\bsq = \bsh +\bsell
\in\bbZ^{s-1}$, it is easy to show that
\begin{align} \label{eq:sumGg}
\sum_{\substack{g \in \{1:n\} \\ g \mid n}} G^g (h_s, \ell_s)
& = \bigg(\sum_{\bsh \in \bbZ^{s-1}}
 \frac{\beta_{ \supp(\bsh,h_s)}^{\lambda} }{r'(\bsh ,h_s)^{\lambda}} \bigg)
 \bigg( \sum_{\bsq \in \bbZ^{s-1}}
 \frac{\beta_{ \supp( \bsq  , h_s+ \ell_s )} ^{\lambda}}{r'( \bsq  , h_s+ \ell_s )^{\lambda}}\bigg)%
 \notag \\
&=  \begin{cases}
 \displaystyle\frac{1}{\lvert \ell_s \rvert^{\alpha\lambda}} \scrP \scrQ  &  {\textnormal{if } h_s=0 \textnormal{ and } \ell_s \neq 0 }, \vspace{0.1cm} \\
 \displaystyle\frac{1}{\lvert h_s \rvert ^{\alpha\lambda}} \scrQ  \scrP  &  {\textnormal{if } h_s \neq 0 \textnormal{ and } \ell_s = - h_s }, \vspace{0.1cm}\\
 \displaystyle\frac{1}{\lvert h_s \rvert^{\alpha\lambda}}  \frac{1}{\lvert h_s + \ell_s \rvert ^{\alpha\lambda}} \scrQ^2   &       {\textnormal{if } h_s \neq 0 \textnormal{ and } \ell_s \neq - h_s },
 \end{cases}
\end{align}
with the abbreviations
\begin{align*}
\scrP := \sum_{ \setu \subseteq \{1:s-1\}  } \beta_{\setu}^{\lambda} \left( 2 \zeta(\alpha \lambda) \right)^{\lvert \setu \rvert}
\quad \textnormal{and} \quad
 \scrQ := \sum_{ \setu \subseteq \{1:s-1\}  } \beta_{\setu \cup \{s\} }^{\lambda} \left( 2 \zeta(\alpha \lambda) \right)^{\lvert \setu \rvert}.
\end{align*}

We now find an upper bound on $F_1$ as follows:
\begin{align*}
F_1 &:=  \sum_{\substack{g \in \{1:n\} \\ g \mid n}} g
	\sum_{\substack{h_s \in \bbZ \setminus \{0\} \\ \gcd(h_s,n)=g}}
	 G^{g}(h_s,-h_s)
	=   \sum_{\substack{g \in \{1:n\} \\ g \mid n}} g^{1-\alpha \lambda}
	\sum_{\substack{\widetilde{h}_s \in \bbZ \setminus \{0\} \\ \gcd(\widetilde{h}_s,n/g)=1}}
	 G^{g}(\widetilde{h}_s,-\widetilde{h}_s) \\
	 & \hphantom{:} \leq
	 	\sum_{\widetilde{h}_s \in \bbZ \setminus \{0\}}
		 \sum_{\substack{g \in \{1:n\} \\ g \mid n}} G^{g}(\widetilde{h}_s,-\widetilde{h}_s)
	=  \sum_{\widetilde{h}_s \in \bbZ \setminus \{0\}}
		\frac{1}{\lvert \widetilde{h}_s \rvert^{\alpha \lambda}} \scrQ \scrP
	= [2 \zeta(\alpha \lambda)] \scrQ \scrP,
\end{align*}
where the inequality is obtained by dropping the condition
$\gcd(\widetilde{h}_s, n/g)=1$ and using $\alpha \lambda >1$. The second
to last inequality follows from \eqref{eq:sumGg}.

Next we find an upper bound on $F_2$ as follows:
\begin{align*}
F_2 & :=  \sum_{\substack{g \in \{1:n\} \\ g \mid n}} g
  	\sum_{ \substack{h_s \in \bbZ \\ g \mid h_s}}  \sum_{\substack{\ell_s \in \bbZ \setminus \{0, -h_s\} \\ \gcd(\ell_s,n)=g}}
	 G^{g}(h_s,\ell_s) \\
	 &\hphantom{:} = \sum_{\substack{g \in \{1:n\} \\ g \mid n}} g
	 	 \sum_{\substack{\ell_s \in \bbZ \setminus \{0\} \\ \gcd(\ell_s,n)=g}}
		 G^{g}(0,\ell_s)
		+ \sum_{\substack{g \in \{1:n\} \\ g \mid n}} g
  	\sum_{\substack{h_s \in \bbZ \setminus \{0\} \\ g \mid h_s}}  \sum_{\substack{\ell_s \in \bbZ \setminus \{0, -h_s\} \\ \gcd(\ell_s,n)=g}}
	 G^{g}(h_s,\ell_s) \\
	 &\hphantom{:}=\sum_{\substack{g \in \{1:n\} \\ g \mid n}} g^{1- \alpha \lambda }
	 	 \sum_{\substack{ \widetilde{\ell}_s \in \bbZ \setminus \{0\} \\ \gcd( \widetilde{\ell}_s,n/g)=1}}
		 G^{g}(0,\widetilde{\ell}_s) \\
		 & \qquad \quad + \sum_{\substack{g \in \{1:n\} \\ g \mid n}} g^{1- 2\alpha \lambda }
		 \sum_{ \widetilde{h}_s \in \bbZ \setminus \{0\} }
		  \sum_{\substack{\widetilde{\ell}_s \in \bbZ \setminus \{0, - \widetilde{h}_s\} \\ \gcd( \widetilde{\ell}_s,n/g)=1}}
	 G^{g}( \widetilde{h}_s, \widetilde{\ell}_s) \\
	 & \hphantom{:}\leq 	  \sum_{\widetilde{\ell}_s \in \bbZ \setminus \{0\} }
	 \sum_{\substack{g \in \{1:n\} \\ g \mid n}}   G^{g}(0,\widetilde{\ell}_s)
	 +	 	 \sum_{ \widetilde{h}_s \in \bbZ \setminus \{0\} }
		  \sum_{\widetilde{\ell}_s \in \bbZ \setminus \{0, -\widetilde{h}_s\} }
		   \sum_{\substack{g \in \{1:n\} \\ g \mid n}}    G^{g}( \widetilde{h}_s, \widetilde{\ell}_s) \\
	&\hphantom{:} = \sum_{\widetilde{\ell}_s \in \bbZ \setminus \{0\} } \frac{1}{\lvert \widetilde{\ell}_s \rvert^{\alpha \lambda}} \scrP \scrQ
	  +	   \sum_{ \widetilde{h}_s \in \bbZ \setminus \{0\} } \frac{1}{\lvert\widetilde{h}_s \rvert^{\alpha \lambda}}
	     \sum_{\widetilde{\ell}_s \in \bbZ \setminus \{0, - \widetilde{h}_s \} }  \frac{1}{\lvert \widetilde{h}_s + \widetilde{\ell}_s \rvert ^{\alpha \lambda}} \scrQ^2 \\
	   &\hphantom{:}\leq    [2 \zeta(\alpha\lambda)] \scrP \scrQ +
	   		  [2 \zeta(\alpha\lambda)]^2  \scrQ^2,
\end{align*}
where we separated the cases $h_s=0$ and $h_s\ne 0$, and used again
\eqref{eq:sumGg}.

To find an upper bound on $F_3$, we write
\begin{align*}
 F_3
 &:=  \sum_{\substack{g \in \{1:n\} \\ g \mid n}} g
  	\sum_{\substack{h_s \in \bbZ \\ g \nmid h_s}}  \sum_{\substack{\ell_s \in \bbZ \setminus \{0, -h_s\} \\ \gcd(\ell_s,n)=g}}
	 G^{g}(h_s,\ell_s) \\
&\hphantom{:}\leq  \sum_{\substack{g \in \{1:n\} \\ g \mid n}} g
  	\sum_{\substack{h_s \in \bbZ \\ g \nmid h_s}}  \sum_{\substack{\ell_s \in \bbZ \setminus \{0, -h_s\} \\ g \mid \ell_s}}
	 G^{g}(h_s,\ell_s) \\
&\hphantom{:}=\sum_{\substack{g \in \{1:n\} \\ g \mid n}} g
      \sum_{\substack{h_s \in \bbZ \\ g \nmid h_s}}
       \frac{1}{\lvert h_s \rvert^{\alpha \lambda}}
        \sum_{\substack{\ell_s \in \bbZ \setminus \{0, -h_s\} \\ g \mid \ell_s}}
       \frac{1}{\lvert h_s + \ell_s \rvert^{\alpha \lambda}} \, \widetilde{G}^g,
\end{align*}
where
\begin{align*} 
 \widetilde{G}^g :=
 \sum_{\bsh \in \bbZ^{s-1}}
 \sum_{\substack{\bsell \in \bbZ^{s-1} \\ \gcd(\bsell \cdot \bsz, n) = g}}
 \frac{\beta^{\lambda}_{ \supp(\bsh) \cup \{s\} } }{r'(\bsh )^{\lambda}}
\frac{\beta^{\lambda}_{ \supp(\bsh + \bsell ) \cup \{s\} } }{r'(\bsh + \bsell )^{\lambda}}.
\end{align*}
Writing $h_s = pg + q$ with $q$ being congruent to the remainder modulo $g$ and $p \in
\bbZ$ and writing $\ell_s = kg$ with $k \in \bbZ \setminus \{0\}$, we
obtain
\begin{align*}
F_3
&\leq \sum_{\substack{g \in \{1:n\} \\ g \mid n}} g
	\sum_{\substack{q \in \bbZ \setminus \{0\} \\ q = \lceil -\frac{g-1}{2} \rceil } }^{\lceil \frac{g-1}{2} \rceil}
	\sum_{p \in \bbZ}  \frac{1}{\lvert pg+q \rvert^{\alpha \lambda}}
	\sum_{k \in \bbZ   \setminus \{0\} }   \frac{1}{\lvert (p+k)g+q \rvert^{\alpha \lambda}}   \widetilde{G}^g \\
&= \sum_{\substack{g \in \{1:n\} \\ g \mid n}} g
	\sum_{\substack{q \in \bbZ \setminus \{0\} \\ q = \lceil -\frac{g-1}{2} \rceil}}^{\lceil \frac{g-1}{2} \rceil}
	\sum_{p \in \bbZ}  \frac{1}{\lvert pg+q \rvert^{\alpha \lambda}}
	\Bigg(
	\sum_{k' \in \bbZ  }   \frac{1}{\lvert k' g+q \rvert^{\alpha \lambda}}  -  \frac{1}{\lvert pg+q\rvert^{\alpha \lambda}}
	\Bigg)  \widetilde{G}^g \\
&= \sum_{\substack{g \in \{1:n\} \\ g \mid n}} g
	\sum_{\substack{q \in \bbZ \setminus \{0\} \\ q = \lceil -\frac{g-1}{2} \rceil}}^{\lceil \frac{g-1}{2} \rceil}
	\Bigg[
	\Bigg( \sum_{p \in \bbZ} \frac{1}{\lvert pg+q\rvert^{\alpha \lambda}} \bigg)^2
   - \Bigg( \sum_{p \in \bbZ} \frac{1}{\lvert pg+q\rvert^{2\alpha \lambda}} \bigg)
	\Bigg]  \widetilde{G}^g .
\end{align*}

To proceed further we need to obtain an upper bound on the inner sums
over $p\in\bbZ$. For any fixed $g \in \{1\mathbin{:}n\}$ and $q$ such that
$\lceil -\frac{g-1}{2} \rceil \leq q \leq \lceil \frac{g-1}{2} \rceil$
we have $\frac{\lvert q \rvert}{g} \leq
\frac{1}{2}$. Moreover, for $p \in \bbZ$ and $p \neq 0$, by the triangle
inequality we have
\begin{align*}
\left\lvert 1 + \frac{q}{pg}\right\rvert \geq 1 - \frac{\lvert q \rvert}{\lvert p \rvert g} \geq 1-\frac{\lvert q \rvert}{g} \geq   \frac{1}{2}.
\end{align*}
Thus
\begin{align*} 
\sum_{p\in \bbZ} \frac{1}{\lvert pg +q \rvert^{\alpha \lambda}}
&= \sum_{p\in \bbZ \setminus \{0\} } \frac{1}{\lvert pg +q \rvert^{\alpha \lambda}} + \frac{1}{\lvert q \rvert^{\alpha \lambda}}
= \sum_{p\in \bbZ \setminus \{0\} } \frac{1}{\lvert pg \rvert^{\alpha \lambda} \lvert 1+\frac{q}{pg} \rvert ^{\alpha \lambda}} + \frac{1}{\lvert q \rvert^{\alpha \lambda}}  \notag\\
&\leq \frac{2^{\alpha \lambda}}{g^{\alpha \lambda}} \sum_{p\in \bbZ \setminus \{0\} }  \frac{1}{\lvert p \rvert^{\alpha \lambda}} + \frac{1}{\lvert q \rvert^{\alpha \lambda}}
 = \frac{2^{\alpha \lambda }}{g^{\alpha \lambda}} [ 2\zeta(\alpha \lambda) ] + \frac{1}{\lvert q \rvert^{\alpha \lambda}},
\end{align*}
and this leads to
\begin{align*}
F_3
& \leq  \sum_{\substack{g \in \{1:n\} \\ g \mid n}} g
	\sum_{ \substack{q \in \bbZ \setminus \{0\} \\ q = \lceil -\frac{g-1}{2} \rceil}}^{\lceil \frac{g-1}{2} \rceil}
	\Bigg[
	\Bigg( \frac{2^{\alpha \lambda }}{g^{\alpha \lambda}} [2\zeta(\alpha \lambda) ] + \frac{1}{ \lvert q \rvert^{\alpha \lambda}} \Bigg)^2
    - \frac{1}{\lvert q \rvert^{2 \alpha \lambda}}
	\Bigg]  \widetilde{G}^g \\
&=  \sum_{\substack{g \in \{1:n\} \\ g \mid n}} g
	\sum_{\substack{q \in \bbZ \setminus \{0\} \\ q = \lceil -\frac{g-1}{2} \rceil}}^{\lceil \frac{g-1}{2} \rceil}
	\Bigg[
	\frac{2^{2\alpha \lambda }}{g^{2\alpha \lambda}}  [2 \zeta(\alpha \lambda)]^2
	+ \frac{2^{\alpha \lambda +1}}{g^{\alpha \lambda}} [2\zeta(\alpha \lambda)] \frac{1}{\lvert q \rvert^{\alpha \lambda}}
	\Bigg]  \widetilde{G}^g \\
& \leq  \sum_{\substack{g \in \{1:n\} \\ g \mid n}} g
		\Bigg[
		\frac{2^{2\alpha \lambda }}{g^{\alpha \lambda}}  [2\zeta(\alpha \lambda)]^2 + \frac{2^{\alpha \lambda + 1}}{g^{\alpha \lambda}}  [2\zeta(\alpha \lambda)]^2
		\Bigg]  \widetilde{G}^g \\
& \leq \left( 2^{2\alpha \lambda } + 2^{\alpha \lambda + 1} \right)  [2\zeta(\alpha \lambda)]^2
		\sum_{\substack{g \in \{1:n\} \\ g \mid n}}  \widetilde{G}^g \\
& \leq  2^{2\alpha \lambda + 1}   [2\zeta(\alpha \lambda)]^2 \scrQ^2,
 \end{align*}
where with a relabeling of $\bsq = \bsh +\bsell\in\bbZ^{s-1}$, we
used
\begin{align*}
\sum_{\substack{g \in \{1:n\} \\ g \mid n}}  \widetilde{G}^g
&= \sum_{\bsh \in \bbZ^{s-1}}  \frac{\beta^{\lambda}_{ \supp(\bsh) \cup \{s\} } }{r'(\bsh)^{\lambda}}
	\sum_{\bsq \in \bbZ^{s-1}} \frac{\beta^{\lambda}_{ \supp(\bsq) \cup \{s\} } }{r'(\bsq)^{\lambda}} \\
&= \Bigg(
	 \sum_{\setu \subseteq \{1:s-1\}}  \sum_{\substack{\bsh \in \bbZ^{s-1} \\ \supp(\bsh)  = \setu}}
 \frac{\beta_{\setu \cup \{s\}}^{\lambda}}{\prod_{j \in \setu} \lvert h_j \rvert^{\alpha \lambda}}
	\Bigg)^2 \\
&= \Bigg(
	\sum_{\setu \subseteq \{1:s-1\}} \beta_{\setu \cup \{s\} }^{\lambda}
	 \prod_{j \in \setu}
	  \sum_{ h \in \bbZ \setminus \{0\}}
	 \frac{1}{ \lvert h \rvert^{\alpha \lambda}}
	\Bigg)^2
= \scrQ^2.
\end{align*}

Combining the upper bounds on $F_1, F_2$ and $F_3$, with $\kappa:=
2^{2\alpha \lambda + 1} + 1$,
we obtain an upper bound on Avg as follows:
\begin{align*}
{\rm{Avg}}
 &\leq \frac{1}{\varphi(n)} (F_1 + F_2 + F_3)\\
 &\leq \frac{1}{\varphi(n)}
 		\bigg( [2 \zeta(\alpha \lambda)] \scrQ \scrP
 		+    [2 \zeta(\alpha\lambda)] \scrP \scrQ
	   		+ [2 \zeta(\alpha\lambda)]^2  \scrQ^2
			+  2^{2\alpha \lambda + 1}   [2\zeta(\alpha \lambda)]^2 \scrQ^2 \bigg) \\
&= \frac{1}{\varphi(n)}
	\bigg(
	2 [2 \zeta(\alpha\lambda)] \scrP \scrQ  +  \big(2^{2\alpha \lambda + 1} +1 \big)  [2 \zeta(\alpha \lambda)]^2 \scrQ^2
	\bigg) \\
& \leq \frac{\kappa}{\varphi(n)}  \Big( 2 \zeta(\alpha\lambda)  \scrQ  \Big)
\Big( \scrP + 2 \zeta(\alpha\lambda)  \scrQ  \Big).
\end{align*}
Substituting the value of $\scrP$ and $\scrQ$ into the above formula, we
obtain
\begin{align*}
{\rm{Avg}}
& \leq \frac{\kappa}{\varphi(n)}
	\Bigg( 2 \zeta(\alpha\lambda) \sum_{ \setu \subseteq \{1:s-1\}  } \beta_{\setu \cup \{s\} }^{\lambda} [2 \zeta(\alpha \lambda)]^{\lvert \setu \rvert}
   \Bigg) \\
	&\qquad \times \Bigg( \sum_{ \setu \subseteq \{1:s-1\}  } \beta_{\setu }^{\lambda} [2 \zeta(\alpha \lambda) ]^{\lvert \setu \rvert}
	  + 2 \zeta(\alpha\lambda) \sum_{ \setu \subseteq \{1:s-1\}  } \beta_{\setu \cup \{s\} }^{\lambda}
        [2 \zeta(\alpha \lambda) ]^{\lvert \setu \rvert} \Bigg) \\
& = \frac{\kappa}{\varphi(n)}
		\Bigg( \sum_{ s \in \setu \subseteq \{1:s\}  } \beta_{\setu}^{\lambda} [2 \zeta(\alpha \lambda)]^{\lvert \setu \rvert}   \Bigg) \\
		&\qquad \times
  \Bigg( \sum_{ \setu \subseteq \{1:s-1\}  } \beta_{\setu }^{\lambda} [2 \zeta(\alpha \lambda)]^{\lvert \setu \rvert}
	  +\sum_{ s \in \setu \subseteq \{1:s\}  } \beta_{\setu}^{\lambda} [2 \zeta(\alpha \lambda) ]^{\lvert \setu \rvert}  \Bigg) \\
&=\frac{\kappa}{\varphi(n)}
		\Bigg( \sum_{ s \in \setu \subseteq \{1:s\}  } \beta_{\setu}^{\lambda} [2 \zeta(\alpha \lambda)]^{\lvert \setu \rvert}   \Bigg)
		\Bigg( \sum_{ \setu \subseteq \{1:s\}  } \beta_{\setu}^{\lambda} [2 \zeta(\alpha \lambda)]^{\lvert \setu \rvert}   \Bigg) .
\end{align*}
This completes the proof.
\end{proof}

With Lemma~\ref{lem:Sdecomp} and the new Lemma~\ref{lem:thetau}, we
can complete the proof of Theorem~\ref{thm:SUapp} by following the
argument in the proof of \cite[Theorem~3.5]{CKNS20}. (We have a slightly
improved constant $\kappa$ here, and we need to replace $n-1$ by
$\varphi(n)$.)

\section{Worst-case $L_{\infty}$ error with general $n$} \label{sec:WceLinf}

The following lemma gives an upper bound on the worst-case
$L_{\infty}$ error. It is a correction of \cite[Equation~(1.3)]{KWW20}
(cf. \cite[Lemma~1]{KWW09c}) which mistakenly claimed \eqref{wce:Linf1} to
be an equality.

\begin{lemma}\label{lem:Linfty-bound}
Given $n\ge 2$ \textnormal{(}prime or composite\textnormal{)}, $d\ge 1$,
$\alpha>1$, weights $\{\gamma_\setu\}_{\setu\subset\bbN}$, $M >0$, let
$A_n$ be the lattice approximation defined by~\eqref{eq:Af} with
index set \eqref{eq:AdM} and generating vector $\bsz \in \bbZ^d$. An upper
bound on the worst-case $L_{\infty}$ error is
\begin{align}
  e^{\rm wor\mbox{-}app}_{n,d,M}(\bsz;L_{\infty})
  &\le
  \Bigg( \sum_{\bsh \not\in \calA_d(M) } \frac{1}{r(\bsh)}  + 2 \sum_{\bsh \in \calA_d(M) } \sum_{\substack{\bsp \not\in \calA_d(M) \\ (\bsp -\bsh) \cdot \bsz \equiv_n 0}} \frac{1}{r(\bsp)} + {\rm{sum}}(T) \Bigg)^{1/2}
  \label{wce:Linf1}
  \\
  &\le
  \Bigg( \sum_{\bsh \not\in \calA_d(M) } \frac{1}{r(\bsh)} +  3\, {\rm{sum}}(T) \Bigg)^{1/2}
  , \label{ieq:wcelinf}
\end{align}
where 
\begin{align}\label{eq:sumt2}
  {\rm sum}(T)
  :=
  \sum_{\bsh \in \calA_d(M)} \sum_{\substack{\bsp \in \calA_d(M) \\ (\bsp-\bsh)\cdot \bsz  \equiv_n 0}}
				\sum_{\substack{\bsell \in \bbZ^d \setminus \{\bszero, \bsp-\bsh\} \\ \bsell \cdot \bsz \equiv_n 0}}  \frac{1}{r(\bsh + \bsell)}
  ,
\end{align}
with $T$ being a matrix defined in the proof in~\eqref{eq:T}.
\end{lemma}

\begin{proof}
Consider $f \in H_d$ and the lattice approximation~\eqref{eq:Af}, and
recall that the error is given by~\eqref{eq:aperr}. Using the reproducing
property of the kernel~\eqref{eq:ker}, and making use of the definition
of the corresponding inner product~\eqref{eq:ip}, we follow the argument in the
proof of \cite[Lemma~1]{KWW09c} to write
\begin{align} \label{eq:aperr2}
(f-A_n(f))(\bsx)
\,=\, \sum_{\bsh \in \bbZ^d} \left \langle f,\tau_{\bsh} \right \rangle_{d,\alpha,\bsgamma}   \rme^{2 \pi \ri \bsh \cdot {\bsx}}
\,=\,  \bigg \langle f, \sum_{\bsh \in \bbZ^d} \tau_{\bsh} \, \rme^{-2 \pi \ri \bsh \cdot {\bsx}}
\bigg\rangle_{d, \alpha,\bsgamma},
\end{align}
where
\begin{align*}
 \tau_{\bsh}(\bst) =
 \begin{cases}
  \displaystyle
  - \!\!\!\! \sum_{\substack{\bsq \in \bbZ^d \setminus \{\bsh\} \\ (\bsh - \bsq) \cdot \bsz \equiv_n 0}}
 \frac{\rme^{2 \pi \ri \bsq \cdot \bst}}{r(\bsq)} & \mbox{for } \bsh\in\calA_d(M),
 \vspace{2mm} \\
 \displaystyle \frac{\rme^{2 \pi \ri \bsh \cdot \bst}}{r(\bsh)} & \mbox{for } \bsh\not\in\calA_d(M).
 \end{cases}
\end{align*}
We can directly read off the Fourier coefficients of the functions
$\tau_\bsh$: for $\bsell\in\bbZ^d$,
\begin{align*}
  \widehat{\tau_\bsh}(\bsell)
  &=
  \begin{cases}
    \displaystyle
    -\frac{1}{r(\bsell)} & \textnormal{if }
    \bsh \in \calA_d(M), \;
    \bsell \ne \bsh \textnormal{ and } \bsh \cdot \bsz \equiv_n \bsell \cdot \bsz,
    \vspace{2mm} \\
    \displaystyle
    \frac{1}{r(\bsell)}  & \textnormal{if }
    \bsh \notin \calA_d(M) \textnormal{ and }
    \bsell = \bsh,
    \vspace{2mm} \\
    0 & \textnormal{otherwise} .
  \end{cases}
\end{align*}
Based on these we obtain
\begin{align} \label{eq:inp}
  \left \langle \tau_{\bsh},\tau_{\bsp} \right \rangle_{d, \alpha,\bsgamma}
  =
  \begin{cases}
    \displaystyle
    \sum_{\substack{\bsell \in \bbZ^d \setminus \{\bszero, \bsp-\bsh\} \\ \bsell \cdot \bsz \equiv_n 0}}
    \!\!\!\!
    \frac{1}{r(\bsh + \bsell)}
    &
    {\textnormal{if } \bsh, \bsp \in \calA_d(M) \textnormal{ and } \bsp \cdot \bsz  \equiv_n \bsh \cdot \bsz }
    ,
    \vspace{2mm}\\
    \displaystyle
    - \frac{1}{r(\bsp)}
    &\hspace{-1cm}
    {\textnormal{if } \bsh \in \calA_d(M),\, \bsp \not\in \calA_d(M) \textnormal{ and } \bsp \cdot \bsz  \equiv_n \bsh \cdot \bsz }
    ,
    \vspace{2mm}\\
    \displaystyle
    - \frac{1}{r(\bsh)}
    &\hspace{-1cm}
    {\textnormal{if } \bsh \not\in \calA_d(M),\, \bsp \in \calA_d(M) \textnormal{ and } \bsp \cdot \bsz  \equiv_n \bsh \cdot \bsz }
    ,
    \vspace{2mm}\\
    \displaystyle
    \frac{1}{r(\bsh)}
    &\hspace{-1cm}
    {\textnormal{if } \bsh = \bsp \not\in \calA_d(M)}
    ,
    \vspace{2mm}\\
    \displaystyle 0
    &\hspace{-1cm}
    {\textnormal{otherwise} }
    .
  \end{cases}
\end{align}
Here we point out that the second and third cases in \eqref{eq:inp} are
negative, which corrects both \cite[Lemma~1]{KWW09c} and \cite[Equation~(1.2)]{KWW20}.

Applying the Cauchy--Schwarz inequality to \eqref{eq:aperr2}, we obtain
\begin{align*}
  \lvert (f-A_n(f))(\bsx)  \rvert
  \leq
  \|f\|_{d, \alpha,\bsgamma}   \, \,
  \bigg\| \sum_{\bsh \in \bbZ^d} \tau_{\bsh} \,  \rme^{-2 \pi \ri \bsh \cdot {\bsx}} \bigg\|_{d,\alpha,\bsgamma}
  .
\end{align*}
Note that the second norm is with respect to the functions
$\tau_\bsh$ and the right-hand side is thus still a function of~$\bsx$.
Equality is attained when $f(\bst)$ and $ \sum_{\bsh \in \bbZ^d}
\tau_{\bsh}(\bst) \, \rme^{-2 \pi \ri \bsh \cdot {\bsx}} $ are linearly
dependent and hence the upper bound is attainable. An upper bound on the
worst-case $L_{\infty}$ error can hence be obtained as follows
\begin{align}
&e^{\rm wor\mbox{-}app}_{n,d,M}(\bsz;L_{\infty})
= \sup_{\bsx \in [0,1]^d} \bigg\|   \sum_{\bsh \in \bbZ^d} \tau_{\bsh} \,
  \rme^{-2 \pi \ri \bsh \cdot {\bsx}} \bigg\|_{d, \alpha,\bsgamma}
  \nonumber
  \\
&=
  \sup_{\bsx \in [0,1]^d}
  \bigg\langle
    \sum_{\bsh \in \bbZ^d} \tau_{\bsh} \,
      \rme^{-2 \pi \ri \bsh \cdot {\bsx}}
    ,
    \sum_{\bsp \in \bbZ^d} \tau_{\bsp} \,
      \rme^{-2 \pi \ri \bsp \cdot {\bsx}}
  \bigg\rangle_{d, \alpha,\bsgamma}^{1/2}
  \nonumber \\
&=  \sup_{\bsx \in [0,1]^d} \bigg\lvert \sum_{\bsh \in \bbZ^d} \sum_{\bsp \in \bbZ^d}
  \langle \tau_{\bsh}, \tau_{\bsp} \rangle_{d, \alpha,\bsgamma}\,
   \rme^{2 \pi \ri (\bsp-\bsh) \cdot {\bsx}} \bigg \rvert^{1/2} \label{eq:dif1} \\
&\le \sup_{\bsx \in [0,1]^d}  \bigg( \sum_{\bsh \in \bbZ^d} \sum_{\bsp \in \bbZ^d}
  \big\lvert \langle \tau_{\bsh}, \tau_{\bsp} \rangle_{d, \alpha,\bsgamma} \big\rvert\,
  \big\lvert\rme^{2 \pi \ri (\bsp-\bsh) \cdot {\bsx}} \big\rvert \bigg)^{1/2} \notag \\
& = \bigg( \sum_{\bsh \in \bbZ^d} \sum_{\bsp \in \bbZ^d}
  \lvert  \langle \tau_{\bsh}, \tau_{\bsp}  \rangle_{d, \alpha,\bsgamma}  \rvert \bigg)^{1/2}. \label{eq:dif2}
\end{align}
Here we point out that the supremum over $\bsx \in [0,1]^d$ in
\eqref{eq:dif1} cannot be attained by taking $\bsx = \bszero$ as
mistakenly claimed in the proof of \cite[Lemma~1]{KWW09c}, since $\langle
\tau_{\bsh},\tau_{\bsp}\rangle_{d,\alpha,\bsgamma} $ could be negative as shown in~\eqref{eq:inp}.

Define the matrix $T$ by
\begin{align}\label{eq:T}
  T
  &:=
  \begin{bmatrix}
    \; \lvert \langle \tau_{\bsh}, \tau_{\bsp} \rangle_{d, \alpha,\bsgamma} \rvert \;
  \end{bmatrix}_{\bsh, \bsp \in \calA_d(M)}
  .
\end{align}
Let ${\rm{trace}}(T)$ denote the sum of its diagonal elements, and
${\rm{sum}}(T)$ the sum of all its elements, as in~\eqref{eq:sumt2}.
Using \eqref{eq:inp} and \eqref{eq:dif2}, we then obtain the first claimed bound~\eqref{wce:Linf1}. To proof the second bound we follow
\cite{KWW20} to bound the middle term in~\eqref{wce:Linf1}
as
\begin{align*}
\sum_{\bsh \in \calA_d(M) } \sum_{\substack{\bsp \not\in \calA_d(M) \\ (\bsp -\bsh) \cdot \bsz \equiv_n 0}} \frac{1}{r(\bsp)}
& \, \leq \, \sum_{\bsh \in \calA_d(M) } \sum_{\substack{\bsp \in \bbZ^d \setminus \{\bsh\} \\ (\bsp -\bsh) \cdot \bsz \equiv_n 0}} \frac{1}{r(\bsp)} \\
& \,= \, \sum_{\bsh \in \calA_d(M) } \sum_{\substack{\bsell \in \bbZ^d \setminus \{\bszero\} \\ \bsell \cdot \bsz \equiv_n 0}} \frac{1}{r(\bsh+\bsell)}
={\rm{trace}}(T) \leq {\rm{sum}}(T)
.
\end{align*}
This completes the proof.
\end{proof}

There are two parts in the upper bound \eqref{ieq:wcelinf} and we will
evaluate them separately. We will need bounds on the cardinality of
$\calA_d(M)$ for general weights.

\begin{lemma} \label{lem:AdM}
For all $d\ge 1$, $\alpha>1$, weights
$\{\gamma_\setu\}_{\setu\subset\bbN}$, and $M >0$, the cardinality of the
index set \eqref{eq:AdM} satisfies the upper bound
\begin{align*}
  \lvert \calA_d(M) \rvert \,\le\, M^q \, C_{1,d,q,\alpha,\bsgamma},
  \qquad \mbox{for all }\quad q> \tfrac{1}{\alpha},
\end{align*}
where
$$
C_{1,d,q, \alpha, \bsgamma} \,:=\, \sum_{\setu\subseteq\{1:d\}}  \gamma_{\setu}^q \, [2\zeta(\alpha q)]^{\lvert \setu \rvert}.
$$
Moreover, if $M\ge 1$ then we have the lower bound
\begin{align} \label{ieq:cardlow}
  \lvert \calA_d(M) \rvert \ge (\gamma_{\{1\}} M)^{1/\alpha}.
\end{align}
\end{lemma}

\begin{proof}
The upper bound is proved in \cite[Lemma~3.1]{CKNS20}. We proceed to prove
a simple lower bound. Suppose $M\ge 1$. Then since $\gamma_\emptyset = 1$
we have $\bszero \in \calA_d(M)$ and
\begin{align*}
  &\lvert \calA_d(M) \rvert
  =
  \sum_{\bsh \in \calA_d(M)} 1
  =
  1 + \sum_{\emptyset \ne \setu \subseteq \{1:d\}}
  \Big\lvert \{ \bsh_\setu \in (\bbZ \setminus \{0\})^{\lvert\setu\rvert} : \prod_{j \in \setu} \lvert h_j \rvert^\alpha \le \gamma_\setu  M \}
  \Big\rvert.
\end{align*}
Restricting the sum to a single subset $\setu = \{1\}$ gives $\lvert
\calA_d(M) \rvert \ge 1 + 2 \lfloor (\gamma_{\{1\}} M)^{1/\alpha}
\rfloor$. If $\gamma_{\{1\}} M \ge 1$, then $1 + 2 \lfloor (\gamma_{\{1\}}
M)^{1/\alpha} \rfloor \ge 1 + \lfloor (\gamma_{\{1\}} M)^{1/\alpha}
\rfloor \ge (\gamma_{\{1\}} M)^{1/\alpha}$. If $\gamma_{\{1\}} M < 1$, then
$1 + 2 \lfloor (\gamma_{\{1\}} M)^{1/\alpha} \rfloor = 1 > (\gamma_{\{1\}}
M)^{1/\alpha}$. Hence in all cases we have the lower bound
\eqref{ieq:cardlow}.
\end{proof}

The first part in \eqref{ieq:wcelinf} is the truncation error and we
will bound it in the following lemma, which extends
\cite[Lemma~6]{KSW08} from product weights to general weights. Note
that the stronger condition $M\ge 1$ is required.

\begin{lemma} \label{lem:srivXM}
For all $d\ge 1$, $\alpha>1$, weights
$\{\gamma_\setu\}_{\setu\subset\bbN}$, $M \ge 1$, and index set
\eqref{eq:AdM}, we have
\begin{align*} 
 \sum_{\bsh \not\in \calA_d(M) } \frac{1}{r(\bsh)} \leq C_{2,d,\tau, \alpha, \bsgamma}\, M^{-\frac{1-\tau}{\alpha \tau}}
 \qquad\mbox{for all } \tau \in (\tfrac{1}{\alpha},1),
\end{align*}
where
\begin{align*}
C_{2,d,\tau, \alpha, \bsgamma} := ( \gamma_{\{1\}} )^{\frac{\tau-1}{\alpha \tau}}
\frac{\tau}{1-\tau} \bigg( \sum_{\setu \subseteq \{1:d\}} \gamma_{\setu}^{\tau} \, [2 \zeta(\alpha \tau)]^{\lvert \setu \rvert} \bigg)^{1/\tau}.
\end{align*}
\end{lemma}
\begin{proof}
Denote by $\bsh^{(i)}$ for $i=1,2,\ldots$ an ordering of $\bsh \in \bbZ^d$
such that $\frac{1}{r(\bsh^{(1)})} \ge \frac{1}{r(\bsh^{(2)})} \ge
\cdots$.
For all $i$ and $\tau > 1/\alpha$, we then have
\begin{align*}
  \frac{1}{r(\bsh^{(i)})^\tau}
  \le
  \frac1i \sum_{j=1}^i \frac{1}{r(\bsh^{(j)})^\tau}
  \le
  \frac1i \sum_{\bsh \in \bbZ^d}  \frac{1}{ r(\bsh)^{\tau} }
  =
  \frac1i \sum_{\setu \subseteq \{1:d\}} \gamma_\setu^\tau \, [2 \zeta(\alpha \tau)]^{\lvert \setu \rvert}
  .
\end{align*}
Therefore, with also $\tau < 1$, we have
\begin{align}
  \nonumber
  \sum_{\bsh \notin \calA_d(M)} \frac{1}{r(\bsh)}
  &=
  \sum_{i > \lvert \calA_d(M) \rvert} \frac{1}{r(\bsh^{(i)})}
  \le
  \sum_{i > \lvert \calA_d(M) \rvert}
  \bigg( \frac1i \sum_{\setu \subseteq \{1:d\}} \gamma_\setu^\tau \, [2 \zeta(\alpha \tau)]^{\lvert \setu \rvert} \bigg)^{1/\tau}
  \\
  \nonumber
  &\le
  \bigg( \sum_{\setu \subseteq \{1:d\}} \gamma_\setu^\tau \, [2 \zeta(\alpha \tau)]^{\lvert \setu \rvert} \bigg)^{1/\tau}
  \int_{\lvert \calA_d(M) \rvert}^\infty x^{-1/\tau} \rd{x}
  \\
  \label{ieq:wtcerr}
  &=
  \bigg( \sum_{\setu \subseteq \{1:d\}} \gamma_\setu^\tau \, [2 \zeta(\alpha \tau)]^{\lvert \setu \rvert} \bigg)^{1/\tau}
  \frac1{\lvert \calA_d(M) \rvert^{(1-\tau)/\tau}}
  \, \frac{\tau}{1-\tau}
  .
\end{align}
Combining \eqref{ieq:cardlow} and \eqref{ieq:wtcerr} completes the proof.
\end{proof}

The paper \cite{LK19b} studied worst-case $L_{\infty}$ approximation
by a combination of multiple rank-$1$ lattice rules. We remark that our
Lemmas~\ref{lem:AdM} and~\ref{lem:srivXM} can be used to extend the
results in \cite{LK19b} from product weights to general weights.

The second part in our worst-case $L_\infty$ error
bound~\eqref{ieq:wcelinf} is $3 \, {\rm{sum}}(T)$,
which bounds the cubature error. 
We look at two ways of bounding ${\rm{sum}}(T)$, both making use of
the quantity $S_{n,d,\alpha,\bsgamma}(\bsz)$ defined
in~\eqref{eq:Sd}.

\begin{lemma}\label{lem:sumT-bounds}
Given $n \geq 2$ \textnormal{(}prime or composite\textnormal{)}, $d\ge
1$, $\alpha>1$, weights $\{\gamma_\setu\}_{\setu\subset\bbN}$, and
generating vector $\bsz\in\bbZ^d$, the quantity ${\rm{sum}}(T)$ defined by
\eqref{eq:sumt2} satisfies
\begin{align}\label{ieq:Sd1inf}
  {\rm{sum}}(T)
  &\le
  M \, \left\lvert \calA_d(M) \right\rvert \, S_{n,d, \alpha, \bsgamma}(\bsz)
  .
\end{align}
Moreover, if $\alpha > 2$ then
\begin{align}\label{ieq:Sd2inf}
  {\rm{sum}}(T)
  &\le
  M \, \big[ S_{n,d,\frac{\alpha}{2},\sqrt{\bsgamma}}(\bsz) \big]^2
  ,
\end{align}
where 
$S_{n,d, \frac{\alpha}{2},\sqrt{\bsgamma} }(\bsz )$ is defined by the
expression \eqref{eq:Sd} with $\alpha$ replaced by $\alpha/2$ and each
weight $\gamma_\setu$ replaced by $\sqrt{\gamma_\setu}$.
\end{lemma}

\begin{proof}
The first bound \eqref{ieq:Sd1inf} is obtained (as shown
in~\cite{KWW09c}) by bounding the sum over $\bsp$ in \eqref{eq:sumt2} by
the cardinality of $\calA_d(M)$ and then using $1\le\frac{M}{r(\bsh)}$ for
$\bsh \in \calA_d(M)$.

To derive the second bound \eqref{ieq:Sd2inf}, we write
\eqref{eq:sumt2} as (also shown in~\cite{KWW09c})
\begin{align*}
  {\rm sum}(T)
  &=
  \sum_{\bsh \in \calA_d(M)}
  \sum_{\substack{\bsp \in \calA_d(M) \\ \bsp \cdot \bsz \equiv_n \bsh \cdot \bsz}}
  \sum_{\substack{\bsq \in \bbZ^d \setminus \{\bsp, \bsh\} \\ \bsq \cdot \bsz \equiv_n \bsh \cdot \bsz}} \frac{1}{r(\bsq)}
  =
  \sum_{\bsq \in \bbZ^d} \frac{1}{r(\bsq)}
  \bigg( \sum_{\substack{\bsh \in \calA_d(M) \setminus \{\bsq\} \\ \bsh \cdot \bsz \equiv_n \bsq \cdot \bsz}} 1 \bigg)^2
  .
\end{align*}
Now we use that $\sum_k a_k \leq (\sum_k a_k^{1/2})^2$ for all $a_k
\geq 0$ and that $1\le\frac{M^{1/2}}{[r(\bsh)]^{1/2}}$ for $\bsh \in \calA_d(M)$
to obtain, for $\alpha > 2$,
\begin{align*}
  {\rm sum}(T)
  &\le
  \bigg(
  \sum_{\bsq \in \bbZ^d} \frac{1}{[r(\bsq)]^{1/2}}
  \sum_{\substack{\bsh \in \calA_d(M) \setminus \{\bsq\} \\ \bsh \cdot \bsz \equiv_n \bsq \cdot \bsz}} 1
  \bigg)^2
  \\
  &\le
  M
  \bigg(
  \sum_{\bsq \in \bbZ^d} \frac{1}{[r(\bsq)]^{1/2}}
  \sum_{\substack{\bsh \in \calA_d(M) \setminus \{\bsq\} \\ \bsh \cdot \bsz \equiv_n \bsq \cdot \bsz}} \frac{1}{[r(\bsh)]^{1/2}}
  \bigg)^2
  ,
\end{align*}
from which the second bound~\eqref{ieq:Sd2inf} follows by making use of
the definition~\eqref{eq:Sd} since $[r_{d, \alpha, \bsgamma}(\bsh)]^{1/2}
= r_{d,\frac{\alpha}{2}, \sqrt{\bsgamma}}(\bsh)$. The condition $\alpha >
2$ is needed so that the corresponding error bound \eqref{ieq:SUapp} for
$S_{n,d, \frac{\alpha}{2},\sqrt{\bsgamma} }(\bsz)$ is valid.
\end{proof}

We remark that~\eqref{ieq:Sd1inf} was analyzed in \cite{KWW09c} so
$S_{n,d,\alpha,\bsgamma}(\bsz)$ was proposed as one of the CBC search
criteria for $L_\infty$ approximation for product weights and a prime
number of points. Another quantity $X_d(\bsz)$ which depends on the index
set $\calA_d(M)$ was introduced in \cite{KWW09c} to obtain a better
convergence rate when $\alpha> 2$. We show that our second
bound~\eqref{ieq:Sd2inf} leads to the same better rate when $\alpha> 2$,
and has the advantage that it does not involve any index set.

\begin{theorem}\label{thm:errLinf}
Given $n \geq 2$ \textnormal{(}prime or composite\textnormal{)}, $d \geq
1$, $\alpha>1$, and weights
$\bsgamma=\{\gamma_\setu\}_{\setu\subset\bbN}$, consider the lattice
approximation~\eqref{eq:Af} with index set~\eqref{eq:AdM}.
\begin{enumerate}
\item 
    The generating vector $\bsz$ obtained from
    Algorithm~\textup{\ref{alg:CBCapp}} based on the search criterion
    $S_{n,d, \alpha,\bsgamma }(\bsz)$ in~\eqref{eq:Sd} satisfies
\begin{align*}
  e^{\rm wor\mbox{-}app}_{n,d,M}(\bsz;L_{\infty})
  \,=\, \calO\big( \left[S_{n,d,\alpha,\bsgamma}(\bsz)\right]^{\tau \, r_1} \big)
  \,=\, \calO\big( \varphi(n)^{- r_1} \big)
  \quad\mbox{for all } \tau \in (\tfrac{1}{\alpha},1),
\end{align*}
where $M$ is given by \eqref{eq:M1} $($with $n$ sufficiently large
so that $M\ge 1$$)$ and
\begin{align*}
 r_1 := \frac{1-\tau}{2 \tau (1 - \tau + \alpha\tau + \alpha\tau^2)},
 \quad\mbox{which can be arbitrarily close to}\quad
  \frac{\alpha-1}{4}.
\end{align*}

\item 
    For $\alpha>2$, the generating vector $\bsz$ obtained from
    Algorithm~\textup{\ref{alg:CBCapp}} based on the search criterion
    $S_{n,d, \frac{\alpha}{2},\sqrt{\bsgamma} }(\bsz)$, i.e.
    \eqref{eq:Sd} with $\alpha$ replaced by $\alpha/2$ and weights
    $\gamma_\setu$ replaced by $\sqrt{\gamma_\setu}$, satisfies
\begin{align*} 
  e^{\rm wor\mbox{-}app}_{n,d,M}(\bsz;L_{\infty})
  = \calO\big( \left[S_{n,d,\alpha/2,\sqrt{\bsgamma}}(\bsz)\right]^{2\tau\, r_2} \big)
  = \calO\big( \varphi(n)^{- r_2} \big)
  \, \, \mbox{for all } \tau \in
  (\tfrac{1}{\alpha},\tfrac{1}{2}),
  \end{align*}
where $M$ is given by \eqref{eq:M2} $($with $n$ sufficiently large
so that $M\ge 1$$)$ and
\begin{align*}
 r_2 :=\frac{1-\tau}{2 \tau (1 - \tau + \alpha\tau ) },
 \quad\mbox{which can be arbitrarily close to}\quad
  \frac{\alpha-1}{2}  \frac{1}{2 - \frac{1}{\alpha}}.
\end{align*}
\end{enumerate}
The implied constants are independent of $d$ if \eqref{ieq:idpcon}
holds.
\end{theorem}

\begin{proof}
For both claims we start from the error bound~\eqref{ieq:wcelinf} from
Lemma~\ref{lem:Linfty-bound} and combine this with the bounds from
Lemmas~\ref{lem:AdM}, \ref{lem:srivXM}, and \ref{lem:sumT-bounds}.

For the first claim, we use the first bound~\eqref{ieq:Sd1inf} from
Lemma~\ref{lem:sumT-bounds} to obtain
\begin{align} \label{eq:balance1}
 e^{{\rm wor\mbox{-}app}}_{n,d,M}(\bsz;L_{\infty})
  &\,\le\, \bigg( C_{2,d,\tau, \alpha, \bsgamma} \, M^{-\frac{1-\tau}{\alpha \tau}} \,
  	+ \,3 \, M^{q+1} \, C_{1,d,q, \alpha, \bsgamma} \, S_{n,d, \alpha, \bsgamma}(\bsz)
 \bigg)^{1/2}
\end{align}
for all $q \in (\tfrac{1}{\alpha},\infty)$ and $\tau \in
(\tfrac{1}{\alpha},1)$, where $C_{2,d,\tau, \alpha, \bsgamma}$ is as defined in
Lemma~\ref{lem:srivXM} and $C_{1,d,q, \alpha, \bsgamma} $ is as defined in
Lemma~\ref{lem:AdM}. We take $q=\tau$ and choose $M$ to equate the two
terms inside the brackets in \eqref{eq:balance1} to arrive at
\begin{align} \label{eq:M1}
M =
 \bigg(\frac{C_{2,d,\tau, \alpha, \bsgamma}}{3 \, C_{1,d,\tau, \alpha, \bsgamma}} [S_{n,d,\alpha,\bsgamma} (\bsz)]^{-1}
 \bigg)^{\frac{\alpha \tau}{\alpha\tau^2 + \alpha \tau - \tau +1}}.
\end{align}
Provided that $n$ is sufficiently large, \eqref{eq:M1} will satisfy
$M\ge 1$ and this leads to
\begin{align*}
  e^{\rm wor\mbox{-}app}_{n,d,M}(\bsz;L_{\infty})
    \, \leq \,
  \sqrt{2} \,
    \bigg( 3 \,  C_{2,d,\tau, \alpha, \bsgamma}^{\frac{\alpha \tau(1+\tau)}{1-\tau}} \, C_{1,d,\tau, \alpha, \bsgamma} \, S_{n,d, \alpha, \bsgamma}(\bsz )
    \bigg)^{\frac{1- \tau}{2(1-\tau+\alpha\tau + \alpha \tau^2)}} .
\end{align*}
Using Theorem~\ref{thm:SUapp} and taking $\lambda = \tau$, we conclude
that
\begin{align*} 
  e^{\rm wor\mbox{-}app}_{n,d,M}(\bsz;L_{\infty})
   \, \leq \,
   C_{3,d,\tau, \alpha, \bsgamma}\, \varphi(n)^{-\frac{1-\tau}{2 \tau (1 - \tau + \alpha\tau + \alpha\tau^2) }},
\end{align*}
where
\begin{align*} 
C_{3,d,\tau, \alpha, \bsgamma} &:= \sqrt{2} \,
			\bigg[ 3\, \gamma_{\{1\}}^{-\tau-1} \left( 2^{2\alpha\tau+1} +1 \right)^{1/\tau}
  \Big( \frac{\tau}{1-\tau} \Big)^{\frac{\alpha \tau(1+\tau)}{ 1 - \tau }}
			\bigg]^{\frac{1-\tau}{2 (1 - \tau + \alpha\tau +\alpha\tau^2) }}  \notag \\
	& \, \qquad \quad \times \bigg(
  \sum_{\setu\subseteq\{1:d\}} \max(\lvert \setu \rvert,1)\,\gamma_{\setu}^\tau\, [2\zeta(\alpha\tau)]^{\lvert \setu \rvert}
  	\bigg)^{\frac{1}{2\tau} \left(1 + \frac{ 1 - \tau^2}{1 + (\alpha-1)\tau + \alpha \tau^2} \right)},
\end{align*}
which can be bounded independently of $d$ if \eqref{ieq:idpcon}
holds.

For the second claim, we assume that $\alpha>2$. Then from the second
bound~\eqref{ieq:Sd2inf} from Lemma~\ref{lem:sumT-bounds} we obtain
\begin{align*} 
 e^{\rm wor\mbox{-}app}_{n,d,M}(\bsz;L_{\infty})
  \, \leq \, \bigg( C_{2,d,\tau, \alpha, \bsgamma}\, M^{-\frac{1-\tau}{\alpha \tau}}
  + 3 \, M \, \big[S_{n,d, \frac{\alpha}{2}, \sqrt{\bsgamma}}(\bsz)\big]^2 \bigg)^{1/2}
\end{align*}
for all $\tau \in (\tfrac{1}{\alpha},1)$. 
We again equate the two terms inside the brackets to obtain
\begin{align} \label{eq:M2}
 M = \left( \tfrac{1}{3} \, C_{2,d,\tau, \alpha, \bsgamma} \big[S_{n,d, \frac{\alpha}{2}, \sqrt{\bsgamma}}(\bsz)\big]^{-2} \right)^{\frac{\alpha \tau}{\alpha \tau - \tau +1}}.
\end{align}
Provided that $n$ is sufficiently large, \eqref{eq:M2} will satisfy
$M\ge 1$ and this now leads to
\begin{align*}
 e^{\rm wor\mbox{-}app}_{n,d,M}(\bsz;L_{\infty})
 \, \leq \, \sqrt{2} \, \bigg( 3 \, C_{2,d,\tau, \alpha, \bsgamma}^{\frac{\alpha \tau}{ 1 -\tau}} \,
				\big[S_{n,d, \frac{\alpha}{2}, \sqrt{\bsgamma}}(\bsz)\big]^2
 \bigg)^{\frac{1-\tau}{2 ( 1 -\tau + \alpha \tau )}}.
\end{align*}
We can now apply Theorem~\ref{thm:SUapp}, but with $\alpha$ replaced
by $\tilde{\alpha} := \alpha/2$ and $\bsgamma$ replaced by
$\tilde{\bsgamma} := \sqrt{\bsgamma}$, to obtain for all $\lambda \in
(1/\tilde\alpha,1) \,
= (2/\alpha, 1)$,
\begin{align*}
S_{n,d, \tilde{\alpha}, \tilde{\bsgamma}}(\bsz)
& \, \leq \frac{ \big(2^{\alpha  \lambda + 1} +1\big)^{1/\lambda}}{\varphi(n)^{1/\lambda}}
		\bigg( \sum_{ \setu \subseteq \{1:d\}  } \max(\lvert \setu \rvert,1) \, \gamma_{\setu}^{\frac{\lambda}{2}}
   \Big[ 2 \zeta(\tfrac{\alpha \lambda}{2}) \Big]^{\lvert \setu \rvert}   \bigg)
			^{2/\lambda}.
\end{align*}
Taking $\lambda = 2 \tau$ and restricting $\tau < 1/2$, we can thus obtain
\begin{align*}
  e^{\rm wor\mbox{-}app}_{n,d,M}(\bsz;L_{\infty})
   \, \leq \, C_{4,d,\tau, \alpha, \bsgamma} \,
   \varphi(n)^{-\frac{1-\tau}{2 \tau (1 - \tau + \alpha\tau ) }} ,
\end{align*}
where
\begin{align*} 
C_{4,d,\tau, \alpha, \bsgamma}&:= \sqrt{2} \, \Big( 3 \gamma_{\{1\}}^{-1}
   \big( 2^{2\alpha\tau+1} +1 \big)^{1/\tau} \Big)^{\frac{1-\tau}{2 (1 - \tau + \alpha\tau ) }}
			\Big( \frac{\tau}{1-\tau} \Big)^{\frac{\alpha \tau}{2 (1 - \tau + \alpha\tau ) }} \notag \\
			& \quad \qquad \times  \bigg(
  \sum_{\setu\subseteq\{1:d\}} \max(\lvert \setu \rvert,1)\,\gamma_{\setu}^\tau\, [2\zeta(\alpha\tau)]^{\lvert \setu \rvert} \bigg)^{\frac{1}{2\tau} \left(1 + \frac{1-\tau}{\alpha \tau + 1 -\tau } \right)},
\end{align*}
which can be bounded independently of $d$ if \eqref{ieq:idpcon} holds.
\end{proof}

\section{Embedded lattice rules} \label{sec:Emb_lat}

In this section, we apply techniques from \cite{CKN06} to construct good
generating vectors of embedded lattice rules for a range of number of
points. Recall from Algorithm~\textup{\ref{alg:CBCapp}} that the
search criterion for a fixed $n$ is
$T_{n,d,s,\alpha,\bsgamma}(z_1,\ldots,z_s)$ which contributes to the
dimension-wise decomposition of $S_{n,d,\alpha,\bsgamma}(\bsz)$ as in
Lemma~\textup{\ref{lem:Sdecomp}}. In Algorithm~\textup{\ref{alg:CBCemb}}
below, we will construct embedded lattice rules by a mini-max strategy
based on the ratios of $T_{n,d,s,\alpha,\bsgamma}(z_1,\ldots,z_s)$ against
the ``best'' choices of generating vectors for a range of values of $n$.

\begin{algo} \label{alg:CBCemb}
Given $m_2 > m_1 \ge 1$, prime $p$, $d \geq 1$, $\alpha>1$, and
weights $\bsgamma = \{\gamma_\setu\}_{\setu\subset\bbN}$, for each $m
= m_1,\ldots,m_2,$ we obtain the generating vector $\bsz^{(m)} =
(z_1^{(m)},\ldots,z_d^{(m)})$ using Algorithm~\textup{\ref{alg:CBCapp}}
with $n = p^m$ and store the corresponding values of $T_{p^m,d,s,\alpha,\bsgamma}(z_1^{(m)},\ldots,z_s^{(m)})$ for $s = 1,\ldots,d$.

Then we construct the generating vector $\bsz^{\rm emb} =
(z_1^{\rm emb}, \ldots, z_d^{\rm emb})$ as follows: for each $s = 1,
\ldots,d$, with $z_1^{\rm emb}, \ldots, z_{s-1}^{\rm emb}$ fixed, choose
$z_s$ from
$$\bbU_{p^{m_2}} = \left \{ z \in \bbZ \,: \, 1 \leq z\leq p^{m_2}-1 \, \, \textnormal{and} \, \, \gcd(z, p)=1  \right\}, $$
to minimize 
\begin{align} \label{eq:ratioXs}
  X_{p,m_1,m_2,d,s,\alpha,\bsgamma}(z_1^{\rm emb}, \ldots, z_{s-1}^{\rm emb}, z_s)
 \,:=\, \max_{m_1 \le m \le m_2}
  \frac{T_{p^m,d,s,\alpha,\bsgamma}(z_1^{\rm emb}, \ldots, z_{s-1}^{\rm emb}, z_s)}
       {T_{p^m,d,s,\alpha,\bsgamma}(z_1^{(m)},\ldots,z_s^{(m)})}.
\end{align}
\end{algo}

From the definition \eqref{eq:ratioXs} it follows that the generating
vector $\bsz^{\rm emb}$ obtained by Algorithm~\ref{alg:CBCemb} satisfies,
for each $m$ between $m_1$ and $m_2$,
\begin{align} \label{eq:X_S}
 &S_{p^m, d,\alpha,\bsgamma}(\bsz^{\rm emb})
 = \sum_{s=1}^d T_{p^m,d,s,\alpha,\bsgamma}(z_1^{\rm emb}, \ldots, z_{s}^{\rm emb}) \notag \\
 &\le \sum_{s=1}^d X_{p,m_1,m_2,d,s,\alpha,\bsgamma}(z_1^{\rm emb}, \ldots, z_s^{\rm emb}) \,
  T_{p^m,d,s,\alpha,\bsgamma}(z_1^{(m)},\ldots,z_s^{(m)}) \notag \\
 &\le \bigg( \max_{s \in \{1:d\}}
 X_{p,m_1,m_2,d,s,\alpha,\bsgamma}(z_1^{\rm emb}, \ldots, z_s^{\rm emb}) \bigg)
    S_{p^m,d,\alpha,\bsgamma}(\bsz^{(m)}).
\end{align}%
The quantity $\max_{s \in \{1:d\}}
 X_{p,m_1,m_2,d,s,\alpha,\bsgamma}(z_1^{\rm emb}, \ldots, z_s^{\rm emb})$
 therefore indicates how much worse
$\bsz^{\rm emb}$ obtained from Algorithm~\ref{alg:CBCemb} is compared with
$\bsz^{(m)} $ obtained from Algorithm~\ref{alg:CBCapp} for each number of
points $p^m$, $m_1 \le m \le m_2$, due to the reason that the upper bounds
on the worst-case error measured under both $L_2$ and $L_{\infty}$ are in
terms of $S_{n,d, \alpha, \bsgamma}(\bsz )$ or its variants
$S_{n,d,\alpha/2,\sqrt{\bsgamma}}(\bsz)$, see Theorem~\ref{thm:errL2}
and~\ref{thm:errLinf}.

Theorem~\ref{thm:upbX} shows a theoretical upper bound on the
ratio~\eqref{eq:ratioXs}.

\begin{theorem} \label{thm:upbX}
Given $m_2>m_1 \ge 1$, prime $p$, $d \geq 1$, $\alpha>1$, and weights
$\bsgamma = \{\gamma_\setu\}_{\setu\subset\bbN}$, let $\bsz^{\rm emb} =
(z_1^{\rm emb}, \ldots, z_{d}^{\rm emb})$ be obtained from the CBC
construction following Algorithm~\textup{\ref{alg:CBCemb}}. Then the
maximum in \eqref{eq:X_S} is of order arbitrarily close to
\begin{align*}
  (m_2 - m_1 + 1)^\alpha.
\end{align*}%
In other words, we are penalised by a factor of only $(\log N)^{\alpha}$ with $N = p^{m_2}$.
\end{theorem}

\begin{proof}
Following Algorithm~\textup{\ref{alg:CBCemb}}, for  $s = 1, \ldots, d$,
with $z_1^{\rm emb}, \ldots, z_{s-1}^{\rm emb}$ chosen and fixed, the
next choice $z_s^{\rm emb}$ by the algorithm satisfies for any $\lambda
\in (\frac{1}{\alpha},1] $,
\begin{align} \label{eq:upbXd}
 & \big[X_{p,m_1, m_2,d ,s,\alpha,\bsgamma}
 (z_1^{\rm emb}, \ldots, z_{s-1}^{\rm emb}, z_s^{\rm emb})\big]^{\lambda} \nonumber\\
 & \le \frac{1}{\varphi(p^{m_2})} \sum_{z_s \in \bbU_{p^{m_2}}}
 \big[X_{p,m_1, m_2,d ,s,\alpha,\bsgamma}(z_1^{\rm emb}, \ldots, z_{s-1}^{\rm emb}, z_s)\big]^{\lambda} \notag\\
 &\le \frac{1}{\varphi(p^{m_2})} \sum_{z_s \in \bbU_{p^{m_2}}}
    \bigg[ \sum_{m = m_1}^{m_2} \frac{T_{p^m,d,s,\alpha,\bsgamma}(z_1^{\rm emb}, \ldots, z_{s-1}^{\rm emb}, z_s)}
    {T_{p^m,d,s,\alpha,\bsgamma}(z_1^{(m)},\ldots,z_s^{(m)})} \bigg]^{\lambda} \notag\\
&\le
 \sum_{m = m_1}^{m_2} \frac{\frac{1}{\varphi(p^{m_2})}
 \sum_{z_s \in \bbU_{p^{m_2}}}
 \big[T_{p^m,d,s,\alpha,\bsgamma}(z_1^{\rm emb}, \ldots, z_{s-1}^{\rm emb}, z_s)\big]^{\lambda}
 }{\big[T_{p^m,d,s,\alpha,\bsgamma}(z_1^{(m)},\ldots,z_s^{(m)})\big]^{\lambda}},
\end{align}
where in the first inequality we replaced the minimum over $z_s$ by the
average and in the second inequality we replaced the maximum over $m$ by
the sum, while in the third inequality we used $(\sum_k a_k)^{\lambda} \le
\sum_k a_k^{\lambda}$ for $a_k \ge 0$ and then swapped the order of the
sum over $m$ and the average over $z_s$.

The numerator in \eqref{eq:upbXd} is an average of
$[T_{p^m,d,s,\alpha,\bsgamma}(z_1^{\rm emb}, \ldots, z_{s-1}^{\rm emb},
z_s)]^\lambda$ over $z_s\in\bbU_{p^{m_2}}$, which is exactly the same as
the average over $z_s\in\bbU_{p^m}$, since from
\eqref{eq:Tds}--\eqref{eq:theta} we see that the expression only depends
on $z_s$ through the value of $(z_s \bmod p^m)$. Writing $n = p^m$, we
have from Lemma~\ref{lem:thetau} that, for any $z_1,\ldots,z_{s-1}\in
\bbU_n$,
\begin{align} \label{eq:numer}
 & \frac{1}{\varphi(n)} \sum_{z_s \in \bbU_n}
 \big[T_{n,d,s,\alpha,\bsgamma}(z_1, \ldots, z_{s-1}, z_s)\big]^{\lambda} \nonumber\\
 &\le
 \sum_{\setw\subseteq\{s+1:d\}} \!\!\!\!
 [2\zeta(2\alpha)]^{\lambda \lvert\setw\rvert}\, \frac{1}{\varphi(n)} \sum_{z_s \in \bbU_n}
 \big[ \theta_{n,s,\alpha}\big(z_1, \ldots, z_{s-1},z_s;
       \{\gamma_{\setu\cup\setw}\}_{\setu\subseteq\{1:s\}}\big)\big]^{\lambda} \nonumber \\
 &\le
 \frac{\kappa}{\varphi(n)} \sum_{\setw\subseteq\{s+1:d\}} \!\!\!\!
 [2\zeta(2\alpha)]^{\lambda \lvert\setw\rvert}
 \bigg( \sum_{s \in \setu \subseteq \{1:s\}} \!\!\!\!\gamma_{\setu \cup \setw}^{\lambda} [2 \zeta(\alpha\lambda)]^{\lvert\setu\rvert} \bigg)
 \bigg( \sum_{\setu \subseteq \{1:s\}} \!\!\!\!\gamma_{\setu \cup \setw}^{\lambda} [2 \zeta(\alpha\lambda)]^{\lvert\setu\rvert} \bigg) .
\end{align}

For the denominator in \eqref{eq:upbXd}, we can get a rough lower
bound by restricting \eqref{eq:theta} to the terms with $\bsh = \bszero$,
$\bsell = (0,\ldots,0,\ell_s)$, $\ell_s\ne 0$ and $\ell_s\equiv_n 0$,
where $n=p^m$. We obtain for any $z_1,\ldots,z_s\in\bbU_n$,
\begin{align} \label{eq:denom}
T_{n,d,s,\alpha,\bsgamma} \big(z_1,\ldots,z_s\big)
 \ge \frac{2\zeta(\alpha) \gamma_{\{s\}}}{n^\alpha}.
\end{align}

Substituting \eqref{eq:numer} and \eqref{eq:denom} into \eqref{eq:upbXd},
we conclude that the upper bound to 
\ifthenelse{\equal{\mode}{JournalMode}}{
$X_{p,m_1, m_2,d ,s,\alpha,\bsgamma}(z_1^{\rm emb}, \ldots, z_{s-1}^{\rm emb}, z_s^{\rm emb})$ 
}{
\\$X_{p,m_1, m_2,d ,s,\alpha,\bsgamma}(z_1^{\rm emb}, \ldots, z_{s-1}^{\rm emb}, z_s^{\rm emb})$ 
}
 is of the order
\begin{align*} 
 \bigg(\sum_{m=m_1}^{m_2} \frac{1}{\varphi(p^m)} \, (p^m)^{\alpha\lambda} \bigg)^{1/\lambda}
&= \bigg(\frac{p}{p-1}\sum_{m=m_1}^{m_2}  (p^{\alpha\lambda-1})^m \bigg)^{1/\lambda},
\end{align*}
where the order can be arbitrarily close to $(m_2 - m_1 + 1)^{\alpha}$ as
$\lambda$ approaches $1/\alpha$.%
\end{proof}

\section{Numerical results} \label{sec:NumRes}
For different $d$, $\alpha$ and $\bsgamma$, we will compare values of
$S_{n,d,\alpha,\bsgamma}(\bsz)$ for vectors $\bsz$ obtained by
Algorithm~\ref{alg:CBCapp} for numbers of points $n=2^m$ with those for
prime $n$ in Subsection~\ref{sec:pvsnp}, as well as compare values of
$S_{n,d,\alpha,\bsgamma}(\bsz)$ for embedded lattice rules obtained
by Algorithm~\ref{alg:CBCemb} with those for vectors $\bsz$ obtained by
Algorithm~\ref{alg:CBCapp}  against numbers of points $n=2^m$ in
Subsection~\ref{sec:emb}.

We consider some special forms of weights which
are motivated by applications in uncertainty quantification, see, e.g.,
\cite{KKKNS21,GKNSSS15,GKNSS18b,KN16,KSS12,DKLS16,KKS20}.
\begin{enumerate}
\item For \emph{product weights} \cite{SW98, SW01}, there is a
    positive weight parameter $\gamma_j$ associated with each
    coordinate variable $x_j$, and
    \[
    \gamma_{\setu} = \prod_{j\in \setu} \gamma_j
    \quad\mbox{and}\quad \gamma_{\emptyset} = 1.
    \]
\item For \emph{product and order dependent
    \textnormal{(}POD\textnormal{)} weights} \cite{GKNSSS15, GKNSS18b,
    KN16, KSS12},
there are two sequences $\{\gamma_j\}_{j\ge 1}$ and
$\{\Gamma_\ell\}_{\ell\ge 0}$ such that
\[
 \gamma_{\setu} = \Gamma_{\lvert \setu \rvert} \prod_{j\in \setu} \gamma_j
 \quad\mbox{and}\quad \gamma_{\emptyset} = \Gamma_0 = 1.
\]
\item For \emph{smoothness-driven product and order dependent
    \textnormal{(}SPOD\textnormal{)} weights} \cite{DKLS16, KKS20},
    there is a smoothness degree $\sigma \in \bbN$ and sequences
    $\{\Gamma_\ell\}_{\ell\ge 0}$ and $\{\gamma_{j,\nu}\}_{j\ge
    1,1\le\nu\le\sigma}$ such that, with $\lvert \nu_{\setu} \rvert =
    \sum_{j \in \setu} \nu_j$,
    \[
     \gamma_{\setu} = \sum_{\nu_{\setu} \in \{1:\sigma\}^{\lvert \setu \rvert}
    }\Gamma_{\lvert \nu_{\setu} \rvert}  \prod_{j\in \setu}
    \gamma_{j,\nu_j}
   \quad\mbox{and}\quad \gamma_{\emptyset} = \Gamma_0 = 1.
\]
\end{enumerate}

In this section, we consider two different smoothness
parameters $\alpha \in \{2,4\}$ and three different choices of weights:
\begin{itemize}
\item[]{(a)}  Product weights: $\gamma_j = j^{-1.5 \alpha}$;
\item[]{(b)}  POD weights: $\Gamma_{\ell} = \ell!/a^{\ell},$ $\gamma_j
    = a j^{-1.5 \alpha}$;
\item[]{(c)}  SPOD weights: $\sigma = \alpha/2$, $\Gamma_{\ell} =
    \ell!/a^{\ell},$ $\gamma_{j,\nu} = a \, (2 \,  j^{-1.5
    \alpha})^{\nu}$;
\end{itemize}
with the re-scaling parameter $a = (d!)^{1/d}$ for numerical stability.
These weights were chosen in \cite{CKNS21} to ensure that the implied
constant in the error bound is independent of $d$ and such that they
are distinguishable in the plot.

\subsection{Comparison between lattice rules constructed with prime numbers and composite numbers} \label{sec:pvsnp}
In this subsection, we want to see whether the empirical values of the
convergence rates of $S_{n,d,\alpha,\bsgamma}(\bsz)$ with $\bsz$
constructed for composite $n$ will differ from those constructed for prime
$n$ in~\cite{CKNS21}.

We use the fast CBC constructions developed in \cite{CKNS21} for
approximation based on $S_{n,d,\alpha,\bsgamma}(\bsz)$ together with
the techniques from \cite{NC06b} for composite $n$ to implement
Algorithm~\ref{alg:CBCapp}. In Figure~\ref{fig:np_cbc}, we plot the values
of $S_{n,d,\alpha,\bsgamma}(\bsz)$ against the number of points $n =
2^m$ for $m = 9,10,\ldots, 17$, as well as the values of
$S_{n,d,\alpha,\bsgamma}(\bsz)$ against the prime numbers of points
for $n \in \{ 503, 1009, 2003, 4001, 8009, 16007, 32003, 64007, 128021
\}$.

According to Theorem~\ref{thm:SUapp}, the theoretical rate of convergence
of $S_{n,d,\alpha,\bsgamma}(\bsz)$ is $\calO(\varphi(n)^{-\alpha +
\delta}),$ $\delta>0.$ Table~\ref{tab:empcon} lists the empirical rates of
convergence $\calO(n^{-r})$ for the twelve groups of lines in
Figure~\ref{fig:np_cbc}, where all entries are the values of $r$.

Our key observation is that Figure~\ref{fig:np_cbc}, which is
plotted against $n$ being powers of~$2$, effectively coincides with the
figure in \cite{CKNS21} which was plotted against primes. This is
consistent with Theorem~\ref{thm:SUapp}. The empirical rates in
Table~\ref{tab:empcon} exhibit the expected trend between $\alpha = 2$ and
$\alpha =4$.

As in \cite{CKNS21} we note that different values of $d$ do not affect the
empirical rates of convergence, which is consistent with
Theorem~\ref{thm:errL2}. We remark that since the initial approximation
error $\max_{\setu \subseteq \{1:d\}} \gamma_{\setu}^{1/2}$ may be
different, the empirical values of $S_{n,d,\alpha,\bsgamma}(\bsz)$
vary with the dimension $d$ and weight parameters, so the relative heights
of the lines are irrelevant.

In summary, we have demonstrated that both the theory and construction of
lattice algorithms for function approximation extend well from prime $n$
to composite $n$, 
and from product weights to more complicated forms of weights that arise
from practical applications.

\pgfplotsset{ tick label style={font=\small}, label style={font=\small},
legend style={font=\footnotesize}, every axis/.append style={ line
width=1.0 pt, tick style={line width=0.7pt}}
}

\begin{figure}[t]
\centering
\begin{tikzpicture}[scale = 0.75]
\begin{customlegend}[legend columns=3, legend style={draw=none,at={(4.2,-1)},anchor=north, align=left},,
        legend entries={{non-prime product},
        				{non-prime POD},
				{non-prime SPOD},
				{prime product},
				{prime POD},
				{prime SPOD},
                        }]
       \addlegendimage{color=blue, line width=0.8 pt, line legend}
       \addlegendimage{color=magenta, line width=0.8 pt, line legend}
       \addlegendimage{color=teal, line width=0.8 pt, line legend}
        \addlegendimage{color=red, dotted, line width=0.8 pt, line legend}
        \addlegendimage{color=black, dotted, line width=0.8 pt, line legend}
           \addlegendimage{color=orange, dotted, line width=0.8 pt, line legend}
 \end{customlegend}

\begin{loglogaxis}[
xlabel={$n$}, ylabel={$S_{n,d, \alpha,\bsgamma}(\bsz)$},
grid = major,
            ]
\addplot[color=blue,mark=o,]  table {np_pd_a2_d5.txt} node[below, pos = 1.0, color = black] {$\alpha = 2$};
\addplot[color=blue,mark=o,]  table {np_pd_a2_d10.txt};
\addplot[color=blue,mark=o,]  table {np_pd_a2_d20.txt};
\addplot[color=blue,mark=o,]  table {np_pd_a2_d50.txt};
\addplot[color=blue,mark=o,]  table {np_pd_a2_d100.txt};

\addplot[color=blue,mark=triangle,]  table {np_pd_a4_d5.txt};
\addplot[color=blue,mark=triangle,]  table {np_pd_a4_d10.txt};
\addplot[color=blue,mark=triangle,]  table {np_pd_a4_d20.txt};
\addplot[color=blue,mark=triangle,]  table {np_pd_a4_d50.txt};
\addplot[color=blue,mark=triangle,]  table {np_pd_a4_d100.txt};

\addplot[color=magenta,mark=o,]  table {np_pod_a2_d5.txt};
\addplot[color=magenta,mark=o,]  table {np_pod_a2_d10.txt};
\addplot[color=magenta,mark=o,]  table {np_pod_a2_d20.txt};
\addplot[color=magenta,mark=o,]  table {np_pod_a2_d50.txt};
\addplot[color=magenta,mark=o,]  table {np_pod_a2_d100.txt};

\addplot[color=magenta,mark=triangle,]  table {np_pod_a4_d5.txt};
\addplot[color=magenta,mark=triangle,]  table {np_pod_a4_d10.txt};
\addplot[color=magenta,mark=triangle,]  table {np_pod_a4_d20.txt};
\addplot[color=magenta,mark=triangle,]  table {np_pod_a4_d50.txt};
\addplot[color=magenta,mark=triangle,]  table {np_pod_a4_d100.txt};

\addplot[color=teal,mark=o,]  table {np_spod_a2_d5.txt};
\addplot[color=teal,mark=o,]  table {np_spod_a2_d10.txt};
\addplot[color=teal,mark=o,]  table {np_spod_a2_d20.txt};
\addplot[color=teal,mark=o,]  table {np_spod_a2_d50.txt};
\addplot[color=teal,mark=o,]  table {np_spod_a2_d100.txt};

\addplot[color=teal,mark=triangle,]  table {np_spod_a4_d5.txt};
\addplot[color=teal,mark=triangle,]  table {np_spod_a4_d10.txt};
\addplot[color=teal,mark=triangle,]  table {np_spod_a4_d20.txt};
\addplot[color=teal,mark=triangle,]  table {np_spod_a4_d50.txt};
\addplot[color=teal,mark=triangle,]  table {np_spod_a4_d100.txt};

\addplot[color=red,dotted,mark=o,]  table {pr_pd_a2_d5.txt};
\addplot[color=red,dotted,mark=o,]  table {pr_pd_a2_d10.txt};
\addplot[color=red,dotted,mark=o,]  table {pr_pd_a2_d20.txt};
\addplot[color=red,dotted,mark=o,]  table {pr_pd_a2_d50.txt};
\addplot[color=red,dotted,mark=o,]  table {pr_pd_a2_d100.txt};

\addplot[color=red,mark=triangle,dotted]  table {pr_pd_a4_d5.txt} node[below, pos = 1.0, color = black] {$\alpha = 4$};
\addplot[color=red,mark=triangle,dotted]  table {pr_pd_a4_d10.txt};
\addplot[color=red,mark=triangle,dotted]  table {pr_pd_a4_d20.txt};
\addplot[color=red,mark=triangle,dotted]  table {pr_pd_a4_d50.txt};
\addplot[color=red,mark=triangle,dotted]  table {pr_pd_a4_d100.txt};

\addplot[color=black,mark=o,dotted]  table {pr_pod_a2_d5.txt};
\addplot[color=black,mark=o,dotted]  table {pr_pod_a2_d10.txt};
\addplot[color=black,mark=o,dotted]  table {pr_pod_a2_d20.txt};
\addplot[color=black,mark=o,dotted]  table {pr_pod_a2_d50.txt};
\addplot[color=black,mark=o,dotted]  table {pr_pod_a2_d100.txt};

\addplot[color=black,mark=triangle,dotted]  table {pr_pod_a4_d5.txt};
\addplot[color=black,mark=triangle,dotted]  table {pr_pod_a4_d10.txt};
\addplot[color=black,mark=triangle,dotted]  table {pr_pod_a4_d20.txt};
\addplot[color=black,mark=triangle,dotted]  table {pr_pod_a4_d50.txt};
\addplot[color=black,mark=triangle,dotted]  table {pr_pod_a4_d100.txt};

\addplot[color=orange,mark=o,dotted]  table {pr_spod_a2_d5.txt};
\addplot[color=orange,mark=o,dotted]  table {pr_spod_a2_d10.txt};
\addplot[color=orange,mark=o,dotted]  table {pr_spod_a2_d20.txt};
\addplot[color=orange,mark=o,dotted]  table {pr_spod_a2_d50.txt};
\addplot[color=orange,mark=o,dotted]  table {pr_spod_a2_d100.txt};

\addplot[color=orange,mark=triangle,dotted]  table {pr_spod_a4_d5.txt};
\addplot[color=orange,mark=triangle,dotted]  table {pr_spod_a4_d10.txt};
\addplot[color=orange,mark=triangle,dotted]  table {pr_spod_a4_d20.txt};
\addplot[color=orange,mark=triangle,dotted]  table {pr_spod_a4_d50.txt};
\addplot[color=orange,mark=triangle,dotted]  table {pr_spod_a4_d100.txt};

\end{loglogaxis}
\end{tikzpicture}
\caption{The values of $S_{n,d, \alpha,\bsgamma}(\bsz)$ against $n =
2^m$ and prime $n$ for different weights with $\alpha = 2$ (top six
groups) and $\alpha = 4$ (bottom six groups). Each group includes $5$
lines representing $d \in \{5, 10, 20, 50, 100\}.$} \label{fig:np_cbc}
\end{figure}
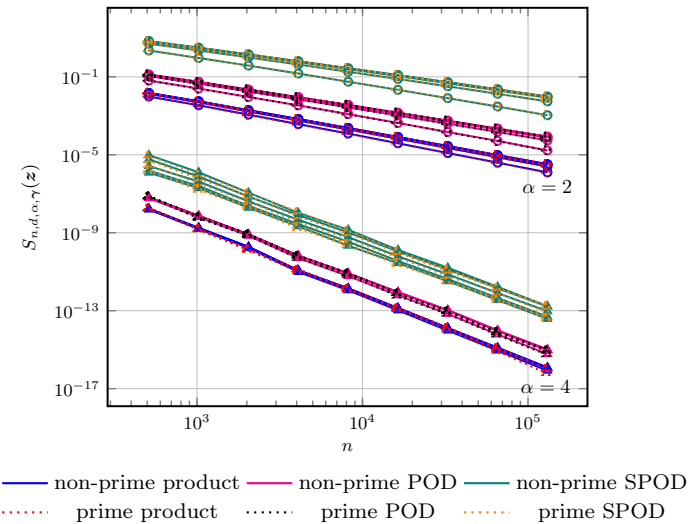

\begin{table}
\centering
\begin{tabular}[htb]{  c | c | c | c | c | c | c }
  \toprule
   &\multicolumn{2}{c|}{Product weights} & \multicolumn{2}{c|}{POD weights} & \multicolumn{2}{c}{SPOD weights} \\
   & $\alpha = 2$ & $\alpha = 4$  & $\alpha = 2$ & $\alpha = 4$ & $\alpha = 2$ & $\alpha = 4$   \\
  \hline \xrowht{10pt}
 prime  & 1.6  & 3.5 &   1.3 & 3.3 & 1.2 & 3.1 \\
  \hline   \xrowht{10pt}
non-prime  &  1.5  & 3.4 &   1.3 & 3.2 & 1.2 & 3.1\\
  \bottomrule
\end{tabular}
\caption{Empirical convergence rates for the twelve groups in
Figure~\ref{fig:np_cbc}} \label{tab:empcon}
\end{table}

\subsection{Comparison between embedded lattice rules and near-optimal lattice rules} \label{sec:emb}

In Figure~\ref{fig:emb_cbc_a2} and~\ref{fig:emb_cbc_a4}, we plot the
values of $S_{2^m,d, \alpha,\bsgamma}(\bsz^{\rm emb})$, of which
$\bsz^{\rm emb}$ is constructed by Algorithm~\ref{alg:CBCemb} for the
range $m = 9, \ldots, 17$, as well as the values of $S_{2^m,d,
\alpha,\bsgamma}(\bsz^{(m)})$ with $\bsz^{(m)}$ constructed by
Algorithm~\ref{alg:CBCapp} for comparison.
Table~\ref{tab:empcon_emb} lists the empirical rates of convergence $\calO(n^{-r})$ for the twelve groups in Figure~\ref{fig:emb_cbc_a2} and~\ref{fig:emb_cbc_a4}, where all entries are
the values of $r$.

\begin{figure}[t]
\centering
\begin{tikzpicture}[scale=0.75]
\begin{customlegend}[legend columns=3, legend style={draw=none,at={(4.2,-1)},anchor=north, align=left},,
        legend entries={{embedded product},
        				{embedded POD},
				{embedded SPOD},
				{near-optimal product},
				{near-optimal POD},
				{near-optimal SPOD},
                        }]
       \addlegendimage{color=blue, dashed, line width=0.8 pt, line legend}
       \addlegendimage{color=magenta, dashed, line width=0.8 pt, line legend}
       \addlegendimage{color=teal, dashed, line width=0.8 pt, line legend}
        \addlegendimage{color=blue, line width=0.8 pt, line legend}
        \addlegendimage{color=magenta, line width=0.8 pt, line legend}
           \addlegendimage{color=teal,  line width=0.8 pt, line legend}
 \end{customlegend}

\begin{loglogaxis}[
xlabel={$n$}, ylabel={$S_{n,d, \alpha,\bsgamma}(\bsz)$},
grid = major, ]
\addplot[color=blue,mark=o, dashed]  table {ext_pd_a2_d5.txt};
\addplot[color=blue,mark=o, dashed]  table {ext_pd_a2_d10.txt};
\addplot[color=blue,mark=o, dashed]  table {ext_pd_a2_d20.txt};
\addplot[color=blue,mark=o, dashed]  table {ext_pd_a2_d50.txt};
\addplot[color=blue,mark=o, dashed]  table {ext_pd_a2_d100.txt};

\addplot[color=magenta,mark=o, dashed]  table {ext_pod_a2_d5.txt};
\addplot[color=magenta,mark=o, dashed]  table {ext_pod_a2_d10.txt};
\addplot[color=magenta,mark=o, dashed]  table {ext_pod_a2_d20.txt};
\addplot[color=magenta,mark=o, dashed]  table {ext_pod_a2_d50.txt};
\addplot[color=magenta,mark=o, dashed]  table {ext_pod_a2_d100.txt};

\addplot[color=teal,mark=o, dashed]  table {ext_spod_a2_d5.txt} ;
\addplot[color=teal,mark=o, dashed]  table {ext_spod_a2_d10.txt};
\addplot[color=teal,mark=o, dashed]  table {ext_spod_a2_d20.txt};
\addplot[color=teal,mark=o, dashed]  table {ext_spod_a2_d50.txt};
\addplot[color=teal,mark=o, dashed]  table {ext_spod_a2_d100.txt};

\addplot[color=blue,mark=o,]  table {np_pd_a2_d5.txt};
\addplot[color=blue,mark=o,]  table {np_pd_a2_d10.txt};
\addplot[color=blue,mark=o,]  table {np_pd_a2_d20.txt};
\addplot[color=blue,mark=o,]  table {np_pd_a2_d50.txt};
\addplot[color=blue,mark=o,]  table {np_pd_a2_d100.txt};

\addplot[color=magenta,mark=o,]  table {np_pod_a2_d5.txt};
\addplot[color=magenta,mark=o,]  table {np_pod_a2_d10.txt};
\addplot[color=magenta,mark=o,]  table {np_pod_a2_d20.txt};
\addplot[color=magenta,mark=o,]  table {np_pod_a2_d50.txt};
\addplot[color=magenta,mark=o,]  table {np_pod_a2_d100.txt};

\addplot[color=teal,mark=o,]  table {np_spod_a2_d5.txt};
\addplot[color=teal,mark=o,]  table {np_spod_a2_d10.txt};
\addplot[color=teal,mark=o,]  table {np_spod_a2_d20.txt};
\addplot[color=teal,mark=o,]  table {np_spod_a2_d50.txt};
\addplot[color=teal,mark=o,]  table {np_spod_a2_d100.txt};

\end{loglogaxis}
\end{tikzpicture}
\caption{The values of $S_{n,d, \alpha,\bsgamma}(\bsz)$ for embedded
lattice rules and near-optimal lattice rules against $n = 2^m$ for
different weights with $\alpha = 2$. Each group includes $5$ lines
representing $d \in \{5, 10, 20, 50, 100\}.$} \label{fig:emb_cbc_a2}
\end{figure}
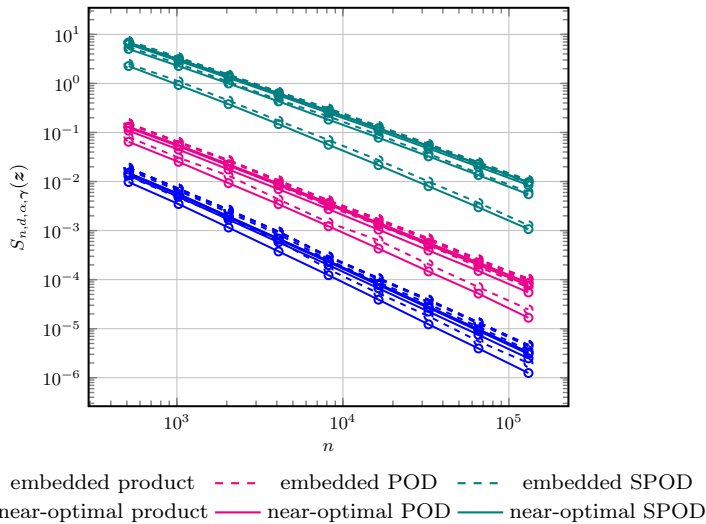

\begin{figure}[t]
\centering
\begin{tikzpicture}[scale = 0.75]
\begin{customlegend}[legend columns=3, legend style={draw=none,at={(4.2,-1)},anchor=north, align=left},,
        legend entries={{embedded product},
        				{embedded POD},
				{embedded SPOD},
				{near-optimal product},
				{near-optimal POD},
				{near-optimal SPOD},
                        }]
       \addlegendimage{color=blue, dashed, line width=0.8 pt, line legend}
       \addlegendimage{color=magenta, dashed, line width=0.8 pt, line legend}
       \addlegendimage{color=teal, dashed, line width=0.8 pt, line legend}
        \addlegendimage{color=blue, line width=0.8 pt, line legend}
        \addlegendimage{color=magenta, line width=0.8 pt, line legend}
           \addlegendimage{color=teal,  line width=0.8 pt, line legend}
 \end{customlegend}

\begin{loglogaxis}[
xlabel={$n$}, ylabel={$S_{n,d, \alpha,\bsgamma}(\bsz)$},
grid = major, ]
\addplot[color=blue,mark=triangle, dashed]  table {ext_pd_a4_d5.txt};
\addplot[color=blue,mark=triangle, dashed]  table {ext_pd_a4_d10.txt};
\addplot[color=blue,mark=triangle, dashed]  table {ext_pd_a4_d20.txt};
\addplot[color=blue,mark=triangle, dashed]  table {ext_pd_a4_d50.txt};
\addplot[color=blue,mark=triangle, dashed]  table {ext_pd_a4_d100.txt};

\addplot[color=magenta,mark=triangle, dashed]  table {ext_pod_a4_d5.txt};
\addplot[color=magenta,mark=triangle, dashed]  table {ext_pod_a4_d10.txt};
\addplot[color=magenta,mark=triangle, dashed]  table {ext_pod_a4_d20.txt};
\addplot[color=magenta,mark=triangle, dashed]  table {ext_pod_a4_d50.txt};
\addplot[color=magenta,mark=triangle, dashed]  table {ext_pod_a4_d100.txt};

\addplot[color=teal,mark=triangle, dashed]  table {ext_spod_a4_d5.txt} ;
\addplot[color=teal,mark=triangle, dashed]  table {ext_spod_a4_d10.txt};
\addplot[color=teal,mark=triangle, dashed]  table {ext_spod_a4_d20.txt};
\addplot[color=teal,mark=triangle, dashed]  table {ext_spod_a4_d50.txt};
\addplot[color=teal,mark=triangle, dashed]  table {ext_spod_a4_d100.txt};

\addplot[color=blue,mark=triangle,]  table {np_pd_a4_d5.txt};
\addplot[color=blue,mark=triangle,]  table {np_pd_a4_d10.txt};
\addplot[color=blue,mark=triangle,]  table {np_pd_a4_d20.txt};
\addplot[color=blue,mark=triangle,]  table {np_pd_a4_d50.txt};
\addplot[color=blue,mark=triangle,]  table {np_pd_a4_d100.txt};

\addplot[color=magenta,mark=triangle,]  table {np_pod_a4_d5.txt};
\addplot[color=magenta,mark=triangle,]  table {np_pod_a4_d10.txt};
\addplot[color=magenta,mark=triangle,]  table {np_pod_a4_d20.txt};
\addplot[color=magenta,mark=triangle,]  table {np_pod_a4_d50.txt};
\addplot[color=magenta,mark=triangle,]  table {np_pod_a4_d100.txt};

\addplot[color=teal,mark=triangle,]  table {np_spod_a4_d5.txt};
\addplot[color=teal,mark=triangle,]  table {np_spod_a4_d10.txt};
\addplot[color=teal,mark=triangle,]  table {np_spod_a4_d20.txt};
\addplot[color=teal,mark=triangle,]  table {np_spod_a4_d50.txt};
\addplot[color=teal,mark=triangle,]  table {np_spod_a4_d100.txt};

\end{loglogaxis}
\end{tikzpicture}
\caption{The values of $S_{n,d, \alpha,\bsgamma}(\bsz)$ for embedded
lattice rules and near-optimal lattice rules against $n = 2^m$ for
different weights with $\alpha = 4$. Each group includes $5$ lines
representing $d \in \{5, 10, 20, 50, 100\}.$} \label{fig:emb_cbc_a4}
\end{figure}
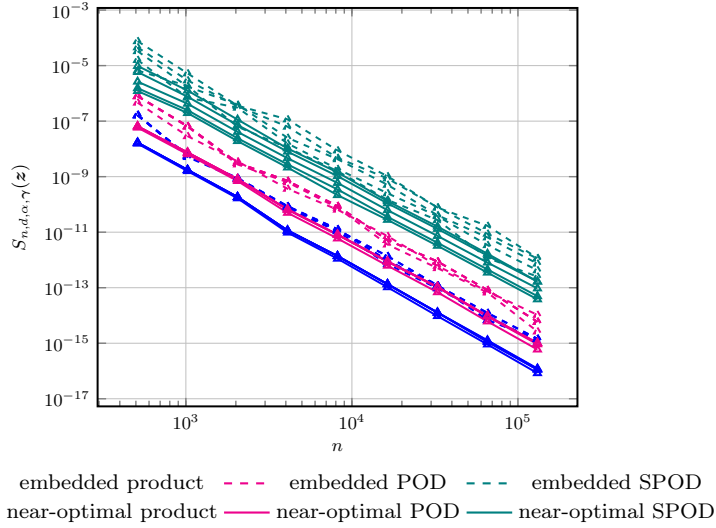

 Given $d=100$, Figure~\ref{fig:Xd} shows the values of  $X_{2, 9,17, d, s, \alpha, \bsgamma}(z_{1}^{\rm emb}, \ldots, z_{s}^{\rm emb})$ against $s =
1,2,\ldots, d$ for product weights, POD weights and SPOD weights with $\alpha = 2$ and $\alpha = 4$.
The maxima of $X_{2, 9,17, d, s, \alpha, \bsgamma}(z_{1}^{\rm emb}, \ldots, z_{s}^{\rm emb})$  over $1 \le s \le 100$ are shown in Table~\ref{tab:max}.

We observe from both Figure~\ref{fig:emb_cbc_a2} and~\ref{fig:emb_cbc_a4}
that each group of lines of $S_{2^m,d, \alpha,\bsgamma}(\bsz^{\rm
emb})$ for embedded lattice rules lie above $S_{2^m,d,
\alpha,\bsgamma}(\bsz^{(m)})$ for near-optimal lattice rules constructed
by Algorithm~\ref{alg:CBCapp}. The groups for $\alpha = 2$ are much closer
than those for $\alpha = 4$. The empirical rates in
Table~\ref{tab:empcon_emb} of embedded lattice rules are close to those of
near-optimal lattice rules and also exhibit the expected trend between
$\alpha = 2$ and $\alpha = 4$.

As in \cite{CKN06}, we see from Figure~\ref{fig:Xd} that, given the quantity $d=100$, 
\ifthenelse{\equal{\mode}{JournalMode}}{
$X_{2, 9,17, d, s, \alpha, \bsgamma}(z_{1}^{\rm emb}, \ldots, z_{s}^{\rm emb})$
}{
$X_{2, 9,17, d, s, \alpha, \bsgamma}(z_{1}^{\rm emb}, \ldots,\\ z_{s}^{\rm emb})$
}
increase initially when $s$ increases, and then decrease from some dimensions and wiggle around some values onward.
In our experiments, Table~\ref{tab:max} shows that, in the case of $\alpha=2$, $\max_{1 \le s \le 100}  X_{2, 9,17, d, s, \alpha, \bsgamma}(z_{1}^{\rm emb}, \ldots, z_{s}^{\rm emb})$
 is at most $2.08$, while in the case of $\alpha = 4$,
$\max_{1 \le s \le 100}  X_{2, 9,17, d, s, \alpha, \bsgamma}(z_{1}^{\rm emb}, \ldots, z_{s}^{\rm emb})$ is at most $25.72$.

\begin{table}
\centering
\begin{tabular}[h!]{  c | c | c | c | c | c | c }
  \toprule
   &\multicolumn{2}{c|}{Product weights} & \multicolumn{2}{c|}{POD weights} & \multicolumn{2}{c}{SPOD weights} \\
   & $\alpha = 2$ & $\alpha = 4$  & $\alpha = 2$ & $\alpha = 4$ & $\alpha = 2$ & $\alpha = 4$   \\
  \hline \xrowht{10pt}
 embedded  & 1.5  & 3.3 &   1.3 & 3.3 &  1.2 & 3.1 \\
  \hline   \xrowht{10pt}
near-optimal  &  1.5  & 3.4 &   1.3 & 3.2 & 1.2 & 3.1\\
  \bottomrule
\end{tabular}
\caption{Empirical convergence rates for the twelve groups in
Figure~\ref{fig:emb_cbc_a2} and~\ref{fig:emb_cbc_a4} } 
\label{tab:empcon_emb}
\end{table}

Considering the worst-case $L_2$ error bound~\eqref{ieq:errMS}, Theorem~\ref{thm:errL2} and \eqref{eq:X_S} show that, 
\begin{align*}
  & e^{\rm wor\mbox{-}app}_{p^m,d,M} (\bsz^{\rm emb} ;L_{2})
  \,\le\, \sqrt{2} \left[ S_{p^m,d, \alpha, \bsgamma}(\bsz^{\rm emb} )   \right]^{1/4} \\
  & \le \,  \left[ \max_{s \in \{1:d\}} X_{p, m_1, m_2 , d, s, \alpha, \bsgamma}(z_1^{\rm emb}, \ldots, z_s^{\rm emb}) \right]^{1/4} \sqrt{2} \left[ S_{p^m,d, \alpha, \bsgamma}(\bsz^{(m)} )   \right]^{1/4} \\
  & \lesssim \left( \frac{ \left[ \max_{s \in \{1:d\}} X_{p, m_1, m_2 , d, s, \alpha, \bsgamma}(z_1^{\rm emb}, \ldots, z_s^{\rm emb}) \right]^{\lambda} }{p^m} \right)^{\frac{1}{4 \lambda}}
  				\mbox{ for all } \lambda \in (1/\alpha, 1],
\end{align*}
where we used that $\varphi(p^m) = p^m (1 - \frac{1}{p}).$
As $\lambda$ can be arbitrarily close to $1/\alpha$, 
\ifthenelse{\equal{\mode}{JournalMode}}{
$\left[ \max_{s \in \{1:d\}} X_{p, m_1, m_2 , d, s, \alpha, \bsgamma}(z_1^{\rm emb}, \ldots, z_s^{\rm emb}) \right]^{1/\alpha}$
}{
\\ $\left[ \max_{s \in \{1:d\}}  X_{p, m_1, m_2 , d, s, \alpha, \bsgamma}(z_1^{\rm emb}, \ldots, z_s^{\rm emb}) \right]^{1/\alpha}$
}
indicates how many times more points are needed for $\bsz^{\rm emb}$ obtained from Algorithm~\ref{alg:CBCemb} to 
achieve the same error bound as $\bsz^{(m)}$ obtained from Algorithm~\ref{alg:CBCapp}
for each $n = p^m, m= m_1, \ldots, m_2$. 
In the case of $\alpha=2$, the embedded lattice rules are at most $1.45$ times worse than the near-optimal lattice rules obtained from Algorithm~\ref{alg:CBCapp}
for specific range of number of points; in the case of $\alpha=4$, this factor increases to $2.26$.

We can consider similarly for the worst-case error measured under $L_{\infty}$ norm, see Theorem~\ref{thm:errLinf}.
In the case of $\alpha=2$, the embedded lattice rules are at most $1.45$ times worse than the near-optimal lattice rules obtained from Algorithm~\ref{alg:CBCapp}
for the specific range of number of points; the same factor holds in the case of $\alpha=4$ 
(in this case, we need to adjust the smoothness parameter to $\tilde{\alpha} = \alpha/2$ and weight parameters to $\tilde{\bsgamma} = \sqrt{\bsgamma}$).

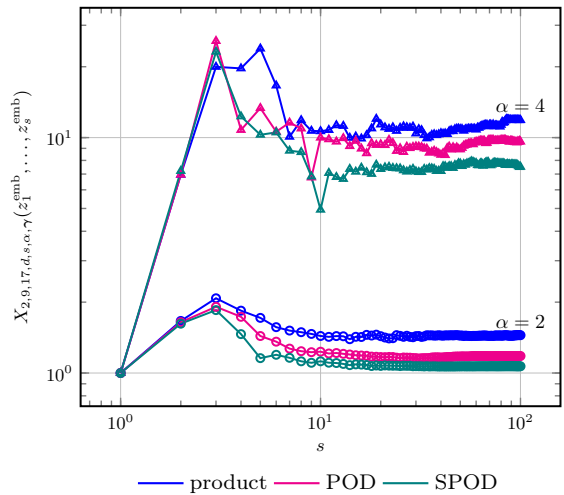
\begin{figure}[t]
\centering
\begin{tikzpicture}[scale=0.75]
\begin{customlegend}[legend columns=3, legend style={draw=none,at={(4.2,-1)},anchor=north, align=left},,
        legend entries={{product},
                        {POD},
                        {SPOD},
                        }]
       \addlegendimage{color=blue, line width=0.8 pt,   line legend}
    
        \addlegendimage{color=magenta,  line width=0.8 pt, line legend}
        
           \addlegendimage{color=teal, line width=0.8 pt, line legend}
         
 \end{customlegend}
\begin{loglogaxis}[
xlabel={$s$}, ylabel={$X_{2, 9,17, d, s, \alpha, \bsgamma}(z_{1}^{\rm emb}, \ldots, z_{s}^{\rm emb})$},
grid = major, ]
\addplot[color=blue,mark=o,]  table {ext_x_pd_a2_d100.txt}  node[above,  pos = 1.0, color = black] {$\alpha = 2$};
\addplot[color=magenta,mark=o,]  table {ext_x_pod_a2_d100.txt};
\addplot[color=teal,mark=o,]  table {ext_x_spod_a2_d100.txt};

\addplot[color=blue,mark=triangle,]  table {ext_x_pd_a4_d100.txt} node[above,  pos = 1.0, color = black] {$\alpha = 4$};
\addplot[color=magenta,mark=triangle,]  table {ext_x_pod_a4_d100.txt};
\addplot[color=teal,mark=triangle,]  table {ext_x_spod_a4_d100.txt};

\end{loglogaxis}
\end{tikzpicture}
\caption{Given $d = 100$, the values of $X_{2, 9,17, d, s, \alpha, \bsgamma}(z_{1}^{\rm emb}, \ldots, z_{s}^{\rm emb})$
 against $s =
1,2,\ldots,d$ for different weights with $\alpha = 2$ (bottom three groups) and $\alpha = 4$ (top three groups).}
\label{fig:Xd}
\end{figure}

\ifthenelse{\equal{\mode}{JournalMode}}{
\begin{table}
\centering
\begin{tabular}[b]{  p{3cm} || p{3cm} | p{3cm}    }
 \toprule
 \multicolumn{3}{c}{$\max_{1 \le s \le 100} X_{2, 9,17, d, s, \alpha, \bsgamma}(z_{1}^{\rm emb}, \ldots, z_{s}^{\rm emb})$} \\ [6pt]
 \hline
  & $\alpha = 2$ & $\alpha = 4$ \\
 \hline
 product   & 2.08    & 23.88 \\
 POD &  1.91  & 25.72 \\
 SPOD & 1.85 & 23.16\\
 \bottomrule
\end{tabular}
\caption{With $d = 100$, the maximum of $X_{2, 9,17, d, s, \alpha, \bsgamma}(z_{1}^{\rm emb}, \ldots, z_{s}^{\rm emb})$ over $s = 1, \ldots, d$ for different weights and $\alpha$ in Figure~\ref{fig:Xd}.}
\label{tab:max}
\end{table}
}{
\begin{table}
\centering
\begin{tabular}[b]{  c || c | c }
 \toprule
 \multicolumn{3}{c}{$\max_{1 \le s \le 100} X_{2, 9,17, d, s, \alpha, \bsgamma}(z_{1}^{\rm emb}, \ldots, z_{s}^{\rm emb})$} \\ [6pt]
 \hline
  & $\alpha = 2$ & $\alpha = 4$ \\
 \hline
 product   & 2.08    & 23.88 \\
 POD &  1.91  & 25.72 \\
 SPOD & 1.85 & 23.16\\
 \bottomrule
\end{tabular}
\caption{With $d = 100$, the maximum of $X_{2, 9,17, d, s, \alpha, \bsgamma}(z_{1}^{\rm emb}, \ldots, z_{s}^{\rm emb})$ over $s = 1, \ldots, d$ for different weights and $\alpha$ in Figure~\ref{fig:Xd}.}
\label{tab:max}
\end{table}
}

\paragraph{Acknowledgements}

We gratefully acknowledge the financial support from the Australian
Research Council (ARC DP210100831) and the Research Foundation -- Flanders
(FWO G091920N).

\bibliographystyle{plain}

\end{document}